\documentclass[11pt, leqno, letterpaper]{amsart}

\usepackage{amsmath,amssymb, amsthm}
\usepackage{url}
 \usepackage{graphicx} 
 \usepackage{bbm, xcolor}
\usepackage{xfrac}
\usepackage{multicol}

\usepackage[margin=1in]{geometry}

\usepackage[english]{babel}

\newtheorem{theorem}{Theorem}

\newtheorem{proposition}[theorem]{Proposition}

\newtheorem{remark}[theorem]{Remark}

\begin{document}

\title[Stochastic growth process driven by an exponential Ornstein-Uhlenbeck process]{Growth rate of a stochastic growth process driven by an exponential 
Ornstein-Uhlenbeck process}
\author{Dan Pirjol}
\address
{School of Business, Stevens Institute of Technology, Hoboken, NJ 07030}
\email{dpirjol@stevens.edu}

\date{June 2021}

\begin{abstract}
We study the stochastic growth process in discrete time
$x_{i+1} = (1 + \mu_i) x_i$ with growth rate  
$\mu_i = \rho e^{Z_i - \frac12 Var(Z_i)}$ proportional to the exponential of an Ornstein-Uhlenbeck (O-U) process $dZ_t = - \gamma Z_t dt + \sigma dW_t$ sampled on a grid of 
uniformly spaced times $\{t_i\}_{i=0}^n$ with time step $\tau$.
Using large deviation theory methods we compute the asymptotic growth rate 
(Lyapunov exponent) $\lambda = \lim_{n\to \infty} \frac{1}{n} \log \mathbb{E}[x_n]$. 
We show that this limit exists, under appropriate scaling of the O-U parameters, and can be expressed as the solution of a variational problem.
The asymptotic growth rate is related to the thermodynamical pressure of a one-dimensional lattice gas with attractive exponential potentials. 
For $Z_t$ a stationary O-U process the lattice gas coincides with a system considered previously by Kac and Helfand. We derive upper and lower bounds on $\lambda$. 
In the large mean-reversion limit $\gamma n \tau \gg 1$ the two bounds
converge and the growth rate is given by a lattice version of the van der Waals equation of state. 
The predictions are tested against numerical simulations of the stochastic growth model.
\end{abstract}

\maketitle

\baselineskip18pt

\section{Introduction}

Stochastic exponential growth processes are widely used to model phenomena
in physics, economics, biology and ecology. 
We consider in this paper a discrete time stochastic growth process of the form
\begin{eqnarray}\label{xndef}
x_{i+1} = (1 + \mu_i ) x_i \,,
\end{eqnarray}
where $\mu_i$ is the growth rate for the period $(t_i,t_{i+1})$, on a grid with time step 
$\tau=t_{i+1}-t_{i}$. If the growth rate is deterministic and proportional to the time step $\mu_i = \mu \tau$, then the process
$x_i$ is a geometric growth process $x_n=x_0 (1+\mu \tau)^n$, which approaches 
in the continuous time limit $\tau \to 0$ with fixed $n \tau = t$ an exponential growth process $x(t) = x_0 e^{\mu t}$.
This type of growth process can be also described as a compounding process, as the 
increment at each time step is proportional to the initial value at the start of the compounding period. 
 
The simplest choice for the randomness of the growth rate $\mu_i$ is a white noise
process, see for example \cite{Tuckwell}. More realistic models allow the
growth rate for each period $\mu_t=f(Z_t)$ to depend on a Markov process 
$Z_t$. The Markov process $Z_t$ represents an environmental variable.
For example, in ecology the variable $x_n$ is the size of a population of a certain species, and the environmental Markov variable $Z_t$ could be the food supply or temperature, which determines the growth rate of the population. The response function $f(Z_t)$ 
modulates the effect of the environmental variable on the growth rate. Typically it is 
an increasing function approaching $f(Z)\to 0$ as $Z\to -\infty$ (adverse environment),
and a finite value $f(Z)\to f_{\rm max}$ as $Z\to +\infty$ (favorable environment).

Stochastic growth processes with Markovian growth rates following a Markov chain model with finite dimensional state space have been widely studied in this context.
The asymptotic growth rates in such models were studied in \cite{JCohen} and more recently in \cite{CohenTL} with application to Taylor's law of fluctuation scaling.
Chapter 14 of Caswell \cite{Caswell} and the monograph by Tulajapurkar \cite{Tuljapurkar} review the literature of models for age- and stage-structured population dynamics with environmental stochasticity.

In this paper we will consider stochastic growth models with growth rates of the form
\begin{eqnarray}\label{1}
\mu_i = \rho e^{Z_i - \frac12 \mbox{var}(Z_i)}
\end{eqnarray}
where $Z_i$ is a Gaussian Markov process sampled on the time grid
$\{t_i\}_{i=0}^n$. The normalization is chosen
such that $\mathbb{E}[\mu_i] = \rho>0$ is a constant. The model (\ref{1}) 
corresponds to an exponential dependence of the growth rate on the environmental 
variable $f(Z,t)=c(t)e^Z$. This choice satisfies two of the conditions for a response function: it is an increasing function and vanishes as $Z\to -\infty$. However since it grows without bound as $Z\to +\infty$, it allows arbitrarily large growth rates.

A similar stochastic growth model was studied in \cite{RMP2,PZ}, where
the Markov process $Z_t= \sigma W_t$ was assumed to be proportional to a 
standard Brownian motion. This model is not very realistic
for applications, where the environmental variable has often a mean-reverting
behavior or a stationary distribution. 
These properties are satisfied by the Ornstein-Uhlenbeck process. 
We will consider here the simplest model of this type, where $Z_t$ 
is an Ornstein-Uhlenbeck (O-U) process 
\begin{eqnarray}\label{OU}
dZ_t  = - \gamma Z_t dt + \sigma dW_t\,.
\end{eqnarray}
For simplicity we assumed that the process is mean-reverting around the zero level, although this assumption can be relaxed by a change of variable $Z_t \to Z_t + \theta$.

We will consider two types of initial condition: stationary O-U process,
and O-U process started at zero $Z_0=0$. The stationary O-U process is appropriate for modeling a steady state environment, while the model with O-U process started at zero is appropriate for situations when the state of the environment is known at the initial time.

In the context of population growth models, the use of Ornstein-Uhlenbeck processes to model the environmental variables has been discussed by Tuckwell \cite{Tuckwell} (Section 2.B). The connection between discrete and continuous time models of this type was discussed by Cumberland and Sykes \cite{Cumberland}. 

In mathematical finance, the stochastic growth process (\ref{1}) represents the bank account compounding interest in discrete time at a stochastic interest rate.
In the Black-Karasinksi model \cite{BlackK} the interest rate is proportional to the exponential of an Ornstein-Uhlenbeck process. In the limit of vanishing mean-reversion this reduces to the Black-Derman-Toy model where the interest rate follows a geometric Brownian motion, and corresponds to the setting considered in \cite{RMP2,PZ}.

In Section~\ref{sec:2} we show that the properties of the stochastic growth process $x_n$ can be mapped to the statistical mechanics of a one-dimensional lattice gas 
with attractive two-body interaction given by an exponential potential which depends on the initial condition for $Z_t$. The expectation $\mathbb{E}[x_n]$ has a natural interpretation as the grand
partition function of a lattice gas, and the growth rate of the expectation of the 
process can be mapped to the thermodynamical pressure of the lattice gas. 
The stationary $Z_t$ case is mapped to a lattice gas studied by Kac and Helfand \cite{KH} as a discrete analog of the Kac-Uhlenbeck-Hemmer system \cite{KUH} of one-dimensional particles interacting by long-range exponential attractive potentials.

In Section \ref{sec:3} we study the growth rates (Lyapunov exponents) of the process in mean $\lambda_n = \frac{1}{n} \log \mathbb{E}[x_n]$. 
We show that $\lim_{n\to \infty} \lambda_n$ exists, provided that the parameters $\sigma,\gamma$ are scaled appropriately with $n$. Using large deviation theory methods 
we give a result for the  (Lyapunov exponent) $\lambda= \frac{1}{n} \log \mathbb{E}[x_n]$ 
expressed as the solution of a variational problem, given by Theorem 1. This is the main result of the paper. 

In Section \ref{sec:4} we study the properties of the variational problem and derive upper and lower bounds on the Lyapunov exponent. These bounds converge in the large mean-reversion limit $a:=\gamma n\tau\gg 1$, which coincides with the mean-field (van der Waals) limit in the analog statistical mechanics system. In this limit the growth rate is described by the van der Waals lattice gas equation of state. 
The growth rate has discontinuous behavior at a critical value of the model
parameters, which corresponds to a phase transition in the lattice gas. 
In Section~\ref{sec:5} the theoretical results are compared with numerical simulations of the model. 

In Section \ref{sec:6} we discuss the relation of our results with existing results in the literature. The stochastic growth process driven by a stationary Ornstein-Uhlenbeck process is mapped to a lattice gas previously considered by Kac and Helfand \cite{KH}. We show that the large mean-reversion limit of our results corresponds to the van der Waals limit in the Kac, Helfand lattice gas, and we reproduce their results. 
An Appendix gives the details of proof of the main result.

%\section{The model}

%In mathematical finance, this is 
%the process followed by a bank account accumulating interest, under stochastic 
%rates given by the exponential of an Ornstein-Uhlenbeck process. 
%This corresponds
%to the Black-Karasinski model of interest rates. 

\subsection{The Ornstein-Uhlenbeck process}

The OU process (\ref{OU}) can be solved exactly as
\begin{equation}\label{OUsol}
Z_{t} = e^{-\gamma t} Z_0 + \sigma \int_0^t e^{-\gamma(t-s)} dW_s \,.
\end{equation}

This process has constant mean $\mathbb{E}[Z_t]=\mathbb{E}[Z_0]$, which we will assume 
to be zero.  The variance is
\begin{equation}\label{Gdef}
\mathbb{E}[Z_t^2] = \mbox{var}(Z_0) e^{-2\gamma t} + G(t)\,,\quad 
G(t) :=  \frac{\sigma^2}{2\gamma} (1 - e^{-2\gamma t}) \,.
\end{equation}

We consider two initial conditions for the Ornstein-Uhlenbeck process $Z_t$:

\textit{Stationary OU process.} The stationary Ornstein-Uhlenbeck process (\ref{OU}) has time independent distribution $Z_t \sim N(0,\frac{\sigma^2}{2\gamma})$ for all $t\geq 0$ and
covariance 
\begin{eqnarray}\label{Zcov2}
\mbox{cov} (Z_{t_i}, Z_{t_j}) = \frac{\sigma^2}{2\gamma}
 e^{-\gamma |t_i-t_j|} \,.
\end{eqnarray}
The covariance is time homogeneous and depends only on the time difference $t_j-t_i$.

\textit{OU process started at zero $Z_0=0$.} 
For any $t>0$, $Z_t$ is normally distributed with zero mean and
time-dependent variance
\begin{eqnarray}
Z_t \sim N\left( 0, \frac{\sigma^2}{2\gamma}(1- e^{-2\gamma t}) \right) \,.
\end{eqnarray}

The covariance is
\begin{eqnarray}\label{Zcov}
\mbox{cov} (Z_i, Z_j) = \frac{\sigma^2}{2\gamma}
\left( e^{-\gamma |t_i-t_j|} - e^{-\gamma (t_i+t_j)} \right)\,.
\end{eqnarray}
%In the large time limit $t_i, t_j \to \infty$ the second term approaches
%zero and the covariance approaches the stationary process result (\ref{Zcov2}).

\section{Relation to one-dimensional lattice gas}
\label{sec:2}

The expectation $\mathbb{E}[x_n]$ of the stochastic growth process $x_{i+1}=(1+\mu_i) x_i$ with multipliers $\mu_i = \rho e^{Z_i - \frac12 Var(Z_i)}$ given by the exponential of an 
O-U process has a simple physical interpretation in terms of the statistical mechanics
of a one-dimensional lattice gas with two-body potentials. 
%as the grand canonical partition function of an equivalent one-dimensional lattice gas.
Following an argument similar to that used in Sec.~5 of \cite{JSP},
this expectation can be expressed as a polynomial in $\rho$ of degree $n-1$
\begin{eqnarray}\label{sumrho}
M_n = \mathbb{E}[x_n] = \sum_{k=0}^n \rho^k \sum_{\mathcal{S}_k}
e^{\sum_{\{ a<b\}\in \mathcal{S}_k} \mbox{cov}(Z_a,Z_b)} =
\sum_{k=0}^{n-1} \rho^k Z_k(n)\,,
\end{eqnarray}
where $\mathcal{S}_k$ is an index set of $k$ indices chosen from the indices
$\{1,2,\cdots , n\}$. The exponent is a sum over all pairs of distinct indices $(a,b)$ in the index set $\mathcal{S}_k$.
The coefficients $Z_k(n)$ have the same form as the canonical partition 
functions of a one-dimensional lattice gas with $n$ sites and $k$ particles, 
interacting by two-body potentials $\beta \varepsilon_{ij} = -\mbox{cov}(Z_i,Z_j)$.

The two-body interaction potentials for the lattice system for the two models considered 
can be read off from the covariance expressions (\ref{Zcov2}) and (\ref{Zcov}) as
\begin{eqnarray}\label{beps1}
&& \beta \varepsilon_{ij}^{\rm I} = 
- \frac{\sigma^2}{2\gamma} e^{-\gamma \tau |i-j|} \hspace{3cm}
\mbox{   (stationary OU process)}\\
\label{beps2}
&& \beta \varepsilon_{ij}^{\rm II} = 
- \frac{\sigma^2}{2\gamma} 
\left( e^{-\gamma \tau |i-j|} - e^{-\gamma \tau (i+j)}\right)\hspace{0.5cm}
\mbox{   (OU process started at zero)}
\end{eqnarray}

The precise definition of the $\beta$ factor will be chosen such that the 
lattice gas interactions are independent of the model parameters, and that 
the lattice gas has a well-defined thermodynamical limit $n \to \infty$.

The expectation $M_n$ is analogous to the grand partition 
function $\mathcal{Z}(\rho,T)$ of a lattice gas with $n$ sites, identifying 
$\rho$ with the fugacity of the lattice gas. This representation holds for any
random multiplicative model of the form (\ref{1}), driven by the exponential
of a Gaussian stochastic process \cite{JSP}. 

The thermodynamical analogy between the random multiplicative model and 
a lattice gas can be constructed by considering the grand canonical
potential at given temperature $T$ and fugacity $\rho$
\begin{eqnarray}
\Omega(\rho, T) = - T \log \mathcal{Z}(\rho, T) = - T \log \mathbb{E}[ x_n]\,.
\end{eqnarray}
The pressure of the lattice gas is
\begin{eqnarray}\label{pdef}
\lim_{n\to \infty} \frac{1}{n} \Omega(\rho,T) = - p(\rho,T)\,.
\end{eqnarray}

The relation
\begin{eqnarray}
d\Omega = - S dT - N d\mu
\end{eqnarray}
can be used to express  the entropy $S$ and number of particles $N$ 
by taking derivatives with respect to the appropriate intensive parameters.
Here $\mu = T \log\rho$ is the chemical potential.

The analog of the thermodynamical pressure (\ref{pdef}) in the random 
multiplicative process
is the exponential growth rate of the expectation $\mathbb{E}[x_n]$.
This  is the Lyapunov exponent, defined as
\begin{eqnarray}\label{plambda}
\lambda := \lim_{n\to \infty} \frac{1}{n} 
\log \mathbb{E}[ x_n ]
\end{eqnarray}
Assuming that the limit exists, the Lyapunov exponent is related to the 
pressure of the analog lattice gas as $\lambda = \frac{1}{T} p$. 
The existence of the limit is non-trivial, and is related to the existence
of the thermodynamical limit in the lattice gas with interactions
(\ref{beps1}), (\ref{beps2}).
This is discussed in the next section, where it is shown that this limit
exists provided that the model parameters $\sigma,\gamma,\tau$ are rescaled
with $n$ appropriately, see Eq.~(\ref{scaling}).

\begin{table}[t!]
\caption{\label{Table:0} 
Equivalence between thermodynamical quantities of the lattice gas and the
parameters of the stochastic growth model. The functional $\Lambda[f]$ and its argument $f(x)=g'(x)$ appear in the statement of Theorem~\ref{thm:OU}.}
\begin{center}
\begin{tabular}{|cc|}
\hline
Stochastic process & Lattice gas \\
\hline
$\mathbb{E}[x_n]$ & Grand partition function $\mathcal{Z}(\rho,\beta)$ \\
$\rho$ & Fugacity $e^{\beta\mu}$ \\
$\Lambda =\lim_{n\to \infty} \frac{1}{n}\log \mathbb{E}[x_n]$ 
   & Landau potential $-\beta \Omega$ \\
$\lambda(\rho,\beta)$ & $\beta p$ with $p$ the gas pressure \\
$f(x)=g'(x)$ & Gas density \\
\hline
\end{tabular}
\end{center}
\end{table}

%The relation (\ref{plambda}) relates the long-run growth rate
%of the average value $\mathbb{E}[ x_n]$ of random multiplicative processes of 
%type (\ref{1}) to the equilibrium statistical mechanics of one-dimensional 
%lattice gases. 

The partial derivatives of $\mathbb{E}[ x_n]$ give the analogs of the particle
density in the lattice gas (average site occupation number)
\begin{eqnarray}
d = \frac{N}{n} = 
- \frac{1}{n} \left( \frac{\partial\Omega }{\partial\mu}\right)_T = \rho 
\left( \frac{\partial \log \mathbb{E}[ x_n]}{\partial \rho}\right)_T
\end{eqnarray}
and the entropy density (entropy per site)
\begin{eqnarray}
s(\rho,\beta) = \frac{1}{n}
\left( \frac{\partial T\log \mathbb{E}[ x_n]}{\partial T}\right)_{\mu} \,.
\end{eqnarray}

The entropy density can be expressed equivalently in terms of partial 
derivatives at fixed fugacity $\rho$ as
\begin{eqnarray}
s(\rho,\beta) &=& \frac{1}{n}
\left( \frac{\partial T\log \mathbb{E}[ x_n]}{\partial T}\right)_\rho +
\frac{1}{n} \rho \log\rho 
\left( \frac{\partial \log \mathbb{E}[ x_n]}{\partial \rho}\right)_T \\
&=&
\frac{1}{n}
\left( \frac{\partial T\log \mathbb{E}[x_n]}{\partial T}\right)_\rho +
d\log\rho  \,.\nonumber
\end{eqnarray}

The behavior of $d,s$ as functions of temperature at fixed fugacity $\rho$ was 
studied in \cite{RMP2} for the particular case of a stochastic driver given by a standard Brownian motion. 
For this case the grand partition function $\mathbb{E}[ x_n]$ could be computed exactly both for finite $n$ (by a recursion method) and in the thermodynamical limit 
$n\to \infty$ \cite{RMP2}.
In this paper we extend these results to the compounding process with 
growth rate given by the exponential of an Ornstein-Uhlenbeck process.

\section{The Lyapunov exponent}
\label{sec:3}

We give in this section an explicit result for the asymptotic growth rate of
the random multiplicative process (\ref{xndef})  in the limit $n\to \infty$ at fixed
\begin{eqnarray}\label{scaling}
\beta = \frac12 \sigma^2 \tau n^2\,,\qquad 
a = \gamma\tau n\,.
\end{eqnarray}

\begin{theorem}\label{thm:OU}
Assume that $x_n$ is a stochastic growth process defined by the recursion
$x_{n+1} = (1 + \rho e^{Z_n - \frac12 Var(Z_n)}) x_n$ driven by the stationary 
Ornstein-Uhlenbeck process $dZ_t = -\gamma Z_t dt + \sigma dW_t$ sampled on discrete times $t_i$ with time step $\tau = t_{i+1}-t_i$. 

The asymptotic growth rate of the expectation $M_n=\mathbb{E}[x_n]$
taken at fixed $\beta=\frac12\sigma^2 \tau n^2, a=\gamma n\tau$ exists and is 
given by solution of the variational problem
\begin{eqnarray}\label{LambdaLim}
&& \lambda(\rho,\beta,a) := \lim_{n\to \infty} \frac{1}{n} \log \mathbb{E}[x_n] =
\sup_{g\in \mathbb{G}} \Lambda[g] 
\end{eqnarray}
with
\begin{eqnarray}\label{LambdaOU}
\Lambda[g] 
%&=& \log\rho g(1) + \frac{\beta}{2a} \left(\int_0^1 e^{-ax} dg(x) \right)^2 
%  \nonumber \\
%& & + \beta \int_0^1 \left( \int_x^1 dg(y) e^{-a(y-x)} \right)^2 dx - 
%  \int_0^1 I(g'(x)) dx \nonumber\\
 =  \log\rho g(1) + \beta \int_0^1 g'(y) g'(z) K(y,z) dy dz -
\int_0^1 I(g'(x)) dx \,.
\end{eqnarray}
The interaction kernel is
\begin{eqnarray}\label{Kdef}
K(y,z) = \frac{1}{2a} e^{-a|y-z|}
\end{eqnarray}
and the entropy function is $I(x) = x \log x + (1-x) \log(1-x)$. 
The supremum in (\ref{LambdaLim}) is taken over functions $g(x): [0,1] \to [0,1]$ in 
\begin{eqnarray}
\mathcal{G} = \{ g: [0,1] \to [0,1]| g(0) = 0\,, g \mbox{ absolutely continuous }, 
   0 \leq g'(x) \leq 1 \}
\end{eqnarray}
\end{theorem}

\begin{proof}
The proof is given in Appendix~\ref{sec:app}.
\end{proof}

The functional $\Lambda$ has a simple physical interpretation as the Landau potential
(grand potential) of a gas of particles in one dimension bounded to the region $x\in [0,1]$
interacting by two-body potentials $V(x,y) = -2K(x,y)$
and maintained at temperature $T=1/\beta$ and chemical potential $\mu\beta = \log\rho$
(equivalently $\rho=e^{\beta\mu}$ is the fugacity)
$$\Lambda[g] \to - \beta \Omega = \mu \beta N - \beta U + S$$
The three terms can be read off from the expression in (\ref{LambdaOU}):

\begin{itemize}

\item $N$ the particle number 
\begin{equation}
N=\int_0^1 f(x) dx
\end{equation}

\item $U$ interaction energy
\begin{equation}
U = -  \int_0^1 K(y,z) f(y) f(z) fy dz = - \frac12 \int_0^1 V(y,z) dy dz
\end{equation}

\item $S$ the entropy of the gas
\begin{equation}
S = - \int_0^1 I(f(x)) dx\,.
\end{equation}

\end{itemize}

\begin{remark}
The result of Theorem~\ref{thm:OU} can be extended to the compounding process driven by an 
OU process started at zero $Z_0=0$ by replacing $K(y,z)\to K_2(y,z)$ in the functional $\Lambda[g]$, with  %in (\ref{K2def}).
\begin{eqnarray}\label{K2def}
K_2(y,z) := \frac{1}{2a} (e^{-a|y-z|} - e^{-a(y+z)}) \,.
\end{eqnarray}

\end{remark}

\begin{proof}
This follows by taking $\mbox{var}(Z_0) = 0$ in $\bar S_n$, see Eq.~(\ref{Sbar}), which has
the effect of dropping the second term in (\ref{zeta}) 
(proportional to $\frac{\beta}{2a}$).
\end{proof}

%%%%%%%%%%%%%%%%%%%%%%%%%%%%%%%%%%%%%%
 \section{General properties and bounds on the Lyapunov exponent}
 \label{sec:4}
 
We discuss here some general properties of the Lyapunov exponent 
$\lambda(\rho,\beta)$ which are immediate consequences of Theorem~\ref{thm:OU} and do not require the solution of the variational problem (\ref{LambdaOU}). 

\begin{proposition}\label{prop:bounds}
Assume the stochastic growth process $x_n$ defined in the statement of  Theorem~\ref{thm:OU}. %driven by a stationary OU process.

(i) The $\beta\to 0$ limit of the Lyapunov exponent is 
$$ \lambda(\rho,0) = \log(\rho+1)$$

(ii) The Lyapunov exponent is bounded from above and below as
\begin{equation}
\bar\lambda(\rho,\beta) \geq \lambda(\rho, \beta) \geq \underline{\lambda}(\rho,\beta) 
\end{equation}
with 
\begin{eqnarray}\label{UBstrong}
&& \bar\lambda(\rho,\beta):=
\sup_{0\leq x \leq 1} \{ \log\rho x + \lambda_0(a) \beta x^2 - I(x) \} \\
\label{LBstrong}
&& \underline{\lambda}(\rho,\beta):=
\sup_{0\leq x \leq 1} \{ \log\rho x + k(a) \beta x^2 - I(x) \} \,.
\end{eqnarray}
where $\lambda_0(a)$ is the largest eigenvalue of the integral equation
$\int_0^1 K(x,y) f(y) dy = \lambda f(x)$ and is given in Eq.~(\ref{lambda0}), and
\begin{equation}
k(a) = \int_0^1 K(y,z) dy dz = \frac{1}{a^3}(a + e^{-a}-1)\,.
\end{equation}

\end{proposition}

\begin{remark}
These results have similar counterparts for the compounding process driven by a
OU process started at zero $Z_0=0$, which are obtained by replacing
\begin{equation}
k(a) \to k_2(a) := \frac{1}{2a^3} \left( 1 + 2a - (2 - e^{-a})^2\right) < k(a) \,.
\end{equation}
As $a \to 0$ the function $k_2(a)$ has the expansion
\begin{equation}
k_2(a) = \frac13 - \frac14 a + O(a^2)
\end{equation}
which reproduces the value $\frac13$ in the limiting case of compounding process driven by a standard Brownian motion. The largest eigenvalue of the equation
$\int_0^1 K_2(y,z) f(z) dz = \lambda f(y)$ is given in equation (\ref{LEV2}).
\end{remark}

\begin{proof}[Proof of Proposition~\ref{prop:bounds}]

(i) Using the convexity property of $I(x)$ we have from Theorem~\ref{thm:OU} the upper bound
\begin{eqnarray}
&& \lambda(\rho,0) = \sup_{g \in \mathcal{G}} \{ \log \rho g(1) - \int_0^1 I(g'(x)) dx \} \\
&&\qquad  \leq  \sup_{0 \leq g(1) \leq 1} \{ \log \rho g(1) -  I(g(1)) \} = \log(\rho+1) \nonumber
\end{eqnarray}
We obtain also a matching lower bound  by choosing $g(x)=x g(1)$
\begin{equation}
\lambda(\rho,0) \geq \sup_{g\in \mathcal{G}} \{ \log\rho g(1) - I(g(1)) \} = \log(\rho+1)
\end{equation}
This concludes the proof of point (i).

(ii) The upper bound $\bar\lambda(\rho,\beta)$ 
follows from the upper bound on the double integral in $\Lambda[g]$
\begin{eqnarray}\label{Kbound}
&& \left| \int_0^1 K(y,z) g'(y) g'(z) dy dz \right| \leq \lambda_0(a) \int_0^1 [g'(y)]^2 dy 
\end{eqnarray}
where $\lambda_0(a)$ the largest eigenvalue of the integral equation
\begin{equation}\label{integralEq}
\int_0^1 K(y,z) f(y) dy = \lambda f(z) \,,\quad K(y,z) = \frac{1}{2a} e^{-a|y-z|}\,.
\end{equation}
This inequality  follows from the maximal 
property of symmetric positive kernels $K(x,y)$
\begin{equation}
\int_0^1 K(x,y) \varphi(x) \varphi(y) dx dy \leq \lambda_0 \int_0^1 \varphi(x)^2 dx
\end{equation}
where $\lambda_0$ is the largest eigenvalue of the integral equation (\ref{integralEq}). See Ch.~III, \S 4.1 in Courant and Hilbert vol. 1 \cite{CourantHilbert}.

%This bound is sufficient to establish relation 
%to a lattice gas of uniformly interacting particles (Hussimi-Temperley model).

The integral equation (\ref{integralEq}) was solved by Mark Kac in \cite{KacBarriers} (see Sec.~4 in this paper). For convenience of reference we state the result and provide a proof below in Proposition~\ref{prop:Kac}. 
The largest eigenvalue of the integral equation (\ref{integralEq}) is 
\begin{equation}\label{lambda0}
\lambda_0(a)  = \frac{1}{a^2(1+y_0^2)}
\end{equation}
where $y_0$ is the smallest non-zero root of the equation
\begin{equation}\label{yeq}
y \tan\left(\frac12 a y\right) = 1
\end{equation}

%This is equivalent with Kac's equation $\tan( a y) = - \frac{2y}{1-y^2}$.

The inequality (\ref{Kbound}) gives an upper 
bound on the functional $\Lambda[g]$
\begin{equation}
\Lambda[g] \leq \log\rho g(1) + \beta \lambda_0(a) \int_0^1 [g'(x)]^2 dx
- \int_0^1 I(g'(x)) dx \,.
\end{equation}

The supremum of the expression on the right-hand side is realized on a function of the form $g(x) = x g(1)$. This follows by expressing the variational problem in terms of $f(x) := g'(x)$ and noting that the Euler-Lagrange equation for $f(x)$ reads 
$\log \rho + 2\beta\lambda_0(a) f(x) = \log\frac{f(x)}{1-f(x)}$. This equation can be written equivalently as $\rho = \frac{d}{1-d} e^{2\beta\lambda_0(a) d}$ which can be recognized as the equation for the density $d$ of a lattice van der Waals gas with coupling $\lambda_0(a)$ at fugacity $\rho$ and inverse temperature $\beta$.
The solution of this equation is a constant $f(x) = d$. An explicit result for $d$ is given below in Proposition~\ref{prop:CW}.

In conclusion we have the upper bound on the Lyapunov exponent
\begin{equation}
\lambda(\rho,\beta) = \sup_g \Lambda[g] \leq \sup_{0 \leq x \leq 1}
\left\{
\log \rho x + \beta \lambda_0(a) x^2 - I(x) \right\} := \bar\lambda(\rho,\beta)\,.
\end{equation}

The lower bound $\underline{\lambda}(\rho,\beta)$ is obtained by taking $g(x) = x g(1)$ in the variational problem (\ref{LambdaLim}).
\end{proof}

\begin{table}[b!]
\caption{\label{Table:0b} 
Numerical values of the constants associated with the upper and lower bounds on the 
Lyapunov exponent of Proposition \ref{prop:bounds} for model 1 (stationary OU process), and model 2 (OU process started at zero). 
%For model I, $y_0(a)$ is the solution of the equation $y \tan(\frac12 a y)=1$,  $\lambda_0(a) = \frac{1}{a^2(1+y_0^2(a))}$ is
%the largest eigenvalue of the integral equation associated with the kernel $K(y,z)$. For model II, $y_0^{(2)}(a)$ is the solution of the equation $\tan(a y)+y=0$ and the largest eigenvalue of the integral equation associated with the kernel $K_2(y,z)$ is $\lambda_0^{(2)}(a) = \frac{1}{a^2[1+(y_0^{(2)}(a))^2]}$.
}
\begin{center}
\begin{tabular}{|c|ccc||ccc|}
\hline
       & \multicolumn{3}{|c|}{Model 1} & \multicolumn{3}{|c|}{Model 2} \\
       \hline
$a$ & $y_0(a)$  & $\lambda_0(a)$ & $k(a)$ 
       & $y_0^{(2)}(a)$  & $\lambda_0^{(2)}(a)$ & $k_2(a)$ \\
\hline
\hline
0.1 & 4.43521 & 4.83768 & 4.83742 
      & 16.3199 & 0.374055 & 0.30946 \\
0.5 & 1.92038 & 0.853269 & 0.852245 
      &  3.6732 & 0.276008 & 0.232973 \\
1.0 & 1.30654 & 0.369405 & 0.36787 
      & 2.0288 & 0.195471 & 0.168091 \\
2.0 & 0.860334 & 0.143664 & 0.141917  
      & 1.1445 & 0.108235 & 0.095189  \\
3.0 & 0.658827 & 0.07748 & 0.075918 
      & 0.8185 & 0.066533 & 0.059198  \\
\hline
\end{tabular}
\end{center}
\end{table}

A tabulation of the parameters $y_0(a),\lambda_0(a), k(a)$ for the model with stationary O-U process (model 1) and their analogs for the model with O-U 
process started at zero (model 2) is given in Table~\ref{Table:0b}.

We give also weaker bounds which are the analogs of the bounds in point (ii) of Proposition 2 in \cite{JSP}.

\begin{remark}
We have the weaker bounds on the Lyapunov exponent 
\begin{equation}\label{UBweak}
k(a) \beta + \log(\rho+1) >
\bar\lambda(\rho,\beta) \geq \lambda(\rho,\beta)
\end{equation}
and 
\begin{equation}\label{LBweak}
\lambda(\rho,\beta) \geq \underline{\lambda}(\rho,\beta) > k(a) \beta + \log\rho
\end{equation}
\end{remark}

\begin{proof}
The weaker upper bound is obtained by bounding the integral in $\Lambda[g]$ as
\begin{eqnarray}
&& \left|\int_0^1 K(y,z) g'(y) g'(z) dy dz \right| \leq \int_0^1 |K(y,z)| |g'(y) | | g'(z) | dy dz \\
&&\qquad \leq \int_0^1 |K(y,z)| dy dz
= k(a) \,. \nonumber
\end{eqnarray}
Using again Jensen's inequality for the integral $\int_0^1 I(g'(x)) dx \geq I(\int_0^1 g'(x) dx) = I(g(1))$
we have
\begin{eqnarray}
\lambda(\rho,\beta) \leq \log\rho g(1) +k(a) \beta - I(g(1)) \leq \log\rho g(1) - I(g(1))
\end{eqnarray}
The lower bound follows by taking $g(x)=x$.

\end{proof}

%\begin{remark}
%The upper bound (\ref{Kbound}) is stronger than the upper bound which follows from $K(y,z) \leq \frac{1}{2a}$. 
%\end{remark}

\subsection{Two integral equations}
\label{sec:4.1}

We discuss in this section the integral equations associated with the upper bounds on the Lyapunov exponent for models with stationary OU process $Z_t$ and OU process started at zero $Z_0=0$. They both have the form
\begin{equation}
\int_0^1 K(y,z) f(y) dy = \lambda f(z)
\end{equation}
where the kernel has the form (\ref{Kdef}) for the stationary OU process and (\ref{K2def}) for the OU process started at zero.

We start with the stationary OU case. The integral equation for this case was solved by Kac in \cite{KacBarriers}. For convenience of the reader we quote the full result
and sketch the proof.

\begin{proposition}[Kac \cite{KacBarriers}]\label{prop:Kac}
Consider the integral equation
\begin{equation}\label{inteq1}
\int_0^1 e^{-a|y-x|} f(y) dy = \eta f(x)
\end{equation}
The eigenvalues are $\eta_j = \frac{2}{(1+y_j^2)a}$ where $\{y_j\}_{j\geq 0}$
are the non-zero solutions of the equation
\begin{equation}\label{eqy}
y \tan\left(\frac12 ay\right) = 1\,.
\end{equation}
The largest eigenvalue $\eta_0$ corresponds to the smallest non-zero
solution of the equation (\ref{eqy}).
\end{proposition}

The eigenvalues of the integral equation $\int_0^1 K(y,z) f(y) dy = \lambda f(z)$ 
with the kernel $K(y,z) = \frac{1}{2a} e^{-|y-z|}$ are obtained by a simple rescaling 
as $\lambda = \frac{1}{2a} \eta$.

\begin{proof}
Change variables as $u = a y$ and $v = a x$. 
Denoting $f(a v) = h(v)$, the integral equation becomes 
\begin{equation}
\int_0^a e^{-|u-v|} h(u) du = a \eta h(v) \,.
\end{equation}
Taking one derivative with respect to $v$ gives
\begin{equation}\label{eqd1}
-\int_0^v e^{-(v-u)} h(u) du + \int_v^a e^{-(u-v)} h(u) du = a \eta h'(v)\,.
\end{equation}
Taking a second derivative gives
\begin{equation}\label{eqd2}
-2 h(v) + \int_0^a e^{-|v-u|} h(u) du = (a \eta - 2) h(v) = a \eta h''(v) \,.
\end{equation}
This equation has the general solution $h(v) = c_1 \sin(y u) + c_2 \cos(y u)$. Upon substitution into (\ref{eqd2}) the eigenvalue is expressed as $a\eta = \frac{2}{1+y^2}$. 

It remains to determine the coefficients $c_{1,2}$ and the constant $y$.
This is done by requiring that $h(v)$ satisfies the original integral equation,
which gives two homogeneous linear equations for $c_{1,2}$. 
Requiring that the determinant of this system vanishes gives the solution
\begin{equation}\label{hsol}
h(u) = \sin(y u) + y \cos(y u)
\end{equation}
where $y$ satisfies the equation (\ref{eqy}).

%Taking $v=0$ in (\ref{eqd1}) gives 
%\begin{equation}
%a\eta h'(0) = \int_0^a e^{-u} h(u) du
%\end{equation}
%Substituting here (\ref{hsol}) gives the eigenvalue $\eta$
%\begin{equation}
%\eta = \frac{2}{(1+y^2)a}
%\end{equation}
%which reproduces the quoted result. 

%{\color{blue} Check that the boundary condition is not satisfied by $e^{\pm y u}$. Done.}

%In conclusion: the eigenvalues of the integral equation (\ref{inteq1}) are
%\begin{equation}\label{lambdaj}
%\eta_j(a) = \frac{2}{(1+y_j^2)a}\,,
%\end{equation}
%where $\{y_j\}_{j\geq 0}$ are the non-zero roots of the equation
%\begin{equation}
%\tan( \frac12 a y) = \frac{1}{y}\,.
%\end{equation}

%Relation to Kac proof: change the integration variable in Kac's equation (5) 
%to $t=a y$, so the upper range of integration becomes $p+q\to a$. Also add a factor of 2 since Kac defines the eigenvalue with a factor of 2 made explicit. 
\end{proof}

\textbf{OU process started at zero.}
Similar results are obtained for the eigenvalues of the kernel $K_2(y,z)$.

\begin{proposition}\label{prop:Kac2}
Consider the integral equation
\begin{equation}\label{inteq2}
\int_0^1 \left( e^{-a|y-x|} - e^{-a(y+x)} \right) f(y) dy = \eta f(x)\,.
\end{equation}

The eigenvalues are $\eta_j = \frac{2}{(1+y_j^2)a}$ where
$\{y_j\}_{j\geq 0}$
are the non-zero solutions of the equation
\begin{equation}\label{eqy2}
\tan\left(ay\right) + y = 0\,.
\end{equation}

The largest eigenvalue $\eta_0$ corresponds to the smallest solution of the equation (\ref{eqy2}).
\end{proposition}

A simple rescaling gives an explicit result for the largest eigenvalue of the integral equation $\int_0^1 K_2(y,z) f(y) dy = \lambda f(z)$
\begin{equation}\label{LEV2}
\lambda_0^{(2)}(a) = \frac{1}{a^2(1 + y_0^{(2)}(a) )}
\end{equation}
where $y_0^{(2)}(a)$ is the smallest root of the equation (\ref{eqy2}).

\begin{proof}
The proof is similar to that of Proposition~\ref{prop:Kac}.
Changing variables as $u = a y$ and $v = a x$ and denoting $f(a v) = h(v)$, we get that $h(v)$ satisfies the same differential equation $(a \lambda - 2) h''(v) = a\lambda h(v)$. The most general solution is $f(x) = c_1 \sin(y a x ) + c_2 \cos(y a x)$ where $\lambda = \frac{2}{a(1+ y^2)}$. 

Substituting the general solution into the integral equation (\ref{inteq2}) 
gives that $c_2=0$ and the solution must have the form 
\begin{equation}
f(x) = C \sin(y a x )
\end{equation}
where the parameter $y$ satisfies the equation (\ref{eqy2}).

\end{proof}

%%%%%%%%%%%%%%%%%%%%%%%%%%%%%%%%%%%%%%
\subsection{Large mean-reversion limit $a \gg 1$}
\label{sec:4.2}

Let us study the asymptotics of $k(a)$ and $\lambda_0(a)$ as $a \gg 1$, which corresponds to the large mean-reversion limit. Recall that the $a$ parameter is related to the mean-reversion parameter $\gamma$ as $a=\gamma n\tau$.

The parameter $k(a)$ appearing in the lower bound of Proposition~\ref{prop:bounds} has the expansion
\begin{equation}\label{kexp}
k(a) = \frac{1}{a^2} - \frac{1}{a^3} + O(e^{-a})\,, \quad a \to \infty
\end{equation}

The parameter $\lambda_0(a)$ appearing in the upper bound has the expansion
\begin{equation}\label{lam0exp}
\lambda_0(a) = \frac{1}{a^2} - \frac{\pi^2}{a^4} + O(a^{-5}) \,, a\to \infty
\end{equation}
This follows from the expansion of the solution of the equation (\ref{yeq}) in the $a\to \infty$ limit
\begin{equation}
y_0(a) = \frac{\pi}{a} - \frac{2\pi}{a(a+2)} + O(a^{-3})
\end{equation}
Substituting into (\ref{lambda0}) gives the expression quoted above.

Comparing (\ref{lam0exp}) and (\ref{kexp}) we note that the constants associated with the upper and lower bounds on $\lambda(\rho,\beta)$ in Proposition \ref{prop:bounds} have the same $a\to \infty$ asymptotics. This implies that the bounds have the same large-$a$ asymptotics and their difference falls off faster than their magnitude. 

This behavior is observed also numerically in Table~\ref{Table:0b} which tabulates $k(a), \lambda_0(a)$ for several values of $a$. This property will play an important
role in Section~\ref{sec:6}, and will be used to show that the large mean-reversion limit corresponds to the van der Waals limit in the analog lattice gas system.

A similar result holds for $\lambda_0^{(2)}(a),k_2(a)$. We have
\begin{equation}\label{k2exp}
k_2(a) = \frac{1}{a^2} - \frac{3}{2a^3} + O(e^{-a})\,. \quad a \to \infty
\end{equation}
A graphical representation of the equation (\ref{eqy2}) shows that the smallest solution is bounded as 
$\frac{\pi}{2a} < y_0^{(2)} < \frac{\pi}{a}$. Thus, as $a\to \infty$ we have $y_0^{(2)}\to 0$. In conclusion we have
\begin{equation}
\lim_{a\to \infty} a^2 \lambda_0^{(2)} = 
\lim_{a\to \infty} a^2 k_2(a) = 1\,.
\end{equation}

\subsection{Mean-field approximation}
\label{sec:4.3}

The lower and upper bounds in Eqs.~(\ref{LBstrong}), (\ref{UBstrong}) can be put 
into a more explicit form, which is related to the solution of the Curie-Weiss mean 
field theory. 
The mean-field theory appears in a wide range of physical and probabilistic models
such as spin systems, lattice gases \cite{Ellis} and random graph models \cite{Aristoff,Radin}, and is the simplest model of phase transitions. 

The upper (\ref{UBstrong}) and lower bounds (\ref{LBstrong}) are related to the function
\begin{equation}\label{fvdW}
\lambda_{\rm vdW}(\rho,\beta) = \sup_{0\leq x \leq 1} \{ \log\rho x + k \beta x^2 - I(x) \}
\end{equation}
with $k\to k(a)$ for the lower bound and $k\to \lambda_0(a)$ for the upper bound.
The function $\lambda_{\rm vdW}(\rho,\beta)$ is related to the thermodynamical pressure of a van der Waals system of interacting particles with interaction parameter $\alpha$ and density $d$
\begin{equation}
\lambda_{\rm vdW} = \beta p_{\rm vdW} = -\log(1-d) - \beta \alpha^2 d^2
\end{equation}
The solution of the extremal problem in Eq.~(\ref{fvdW}) is
summarized in Proposition 7 of \cite{PZ} which we reproduce below for convenience.

\begin{proposition}[Mean-field approximation for the Lyapunov exponent]\label{prop:CW}
The solution of the extremal problem (\ref{fvdW}) is given explicitly by
\begin{equation}\label{lambdavdW0}
\lambda_{\rm vdW}(\rho,\beta) = - k \beta x_*^2 - \log(1-x_*)
\end{equation}
where $x_* = x_*(\rho,\beta)$ is the maximizer in the extremal problem.
The maximizer has the physical interpretation of the particle density $d$
and is given explicitly as:

i) For $\beta > \beta_c = \frac{2}{k}$ the maximizer has a sudden transition between
two solutions
\begin{equation}
x_*(\rho,\beta) = \left\{
\begin{array}{cc}
x_2(\rho,\beta) & \mbox{ if } \log\rho > - k\beta \\
x_1(\rho,\beta) & \mbox{ if } \log\rho < - k \beta \\
\end{array}
\right.
\end{equation}
with $x_{1} < x_2$ the smallest and largest solutions of the equation
\begin{equation}\label{xsteq}
\log\rho = -2k \beta x + \log\frac{x}{1-x} \,.
\end{equation}

ii) For $0 < \beta < \beta_c = \frac{2}{k}$, $x_*(\rho,\beta)$ is the unique solution of 
the equation (\ref{xsteq}).

The mean-field solution (\ref{lambdavdW0}) has a phase transition, manifested as a 
discontinuity of its partial derivatives with respect to its arguments, at points along the
phase transition curve
\begin{equation}
\log\rho + k \beta=0
\end{equation}
with $\beta > \beta_c = \frac{2}{k}$. This curve ends at the critical point
$(\rho_c, \beta_c) = (e^{-2},\frac{2}{k})$. Along this curve the equation
(\ref{xsteq}) has solutions $x_{1,2}=\frac12(1 \pm \Delta(\beta))$ where
$\Delta(\beta)$ is the non-zero positive solution of the equation 
$\Delta = \tanh(\frac12 k \beta \Delta)$. Such a solution exists only for $k\beta>2$.

\end{proposition}

We illustrate the mean-field solution on a numerical example, taking $a=1$ and $\rho=0.025$.
The solutions $x_{1,2}$ of the equation (\ref{xsteq}) for the lower bound $k=k(a)$
are shown in Table~\ref{Table:vdW} for a few
values of the temperature parameter $\beta$ around the transition 
value $\beta_{\rm cr}^{\rm low}=-\frac{\log\rho}{k(a)}=10.0274$. At this point the Lyapunov exponent
is continuous, but its derivative with respect to $\beta$ has a discontinuity. These features are
shown also in graphical form in Figure~\ref{Fig:Lam}.

\begin{figure}[t]
    \centering
\includegraphics[width=2.4in]{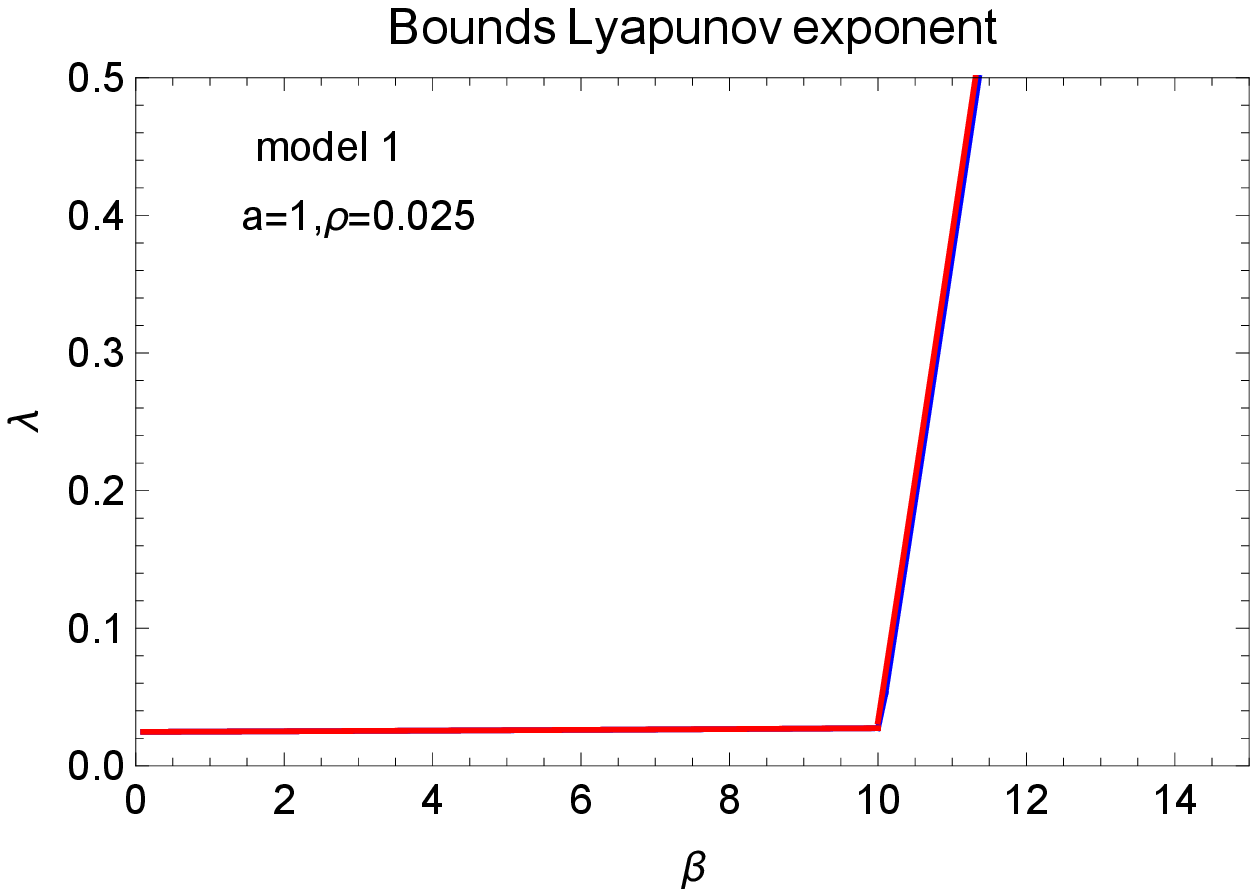}
\includegraphics[width=2.4in]{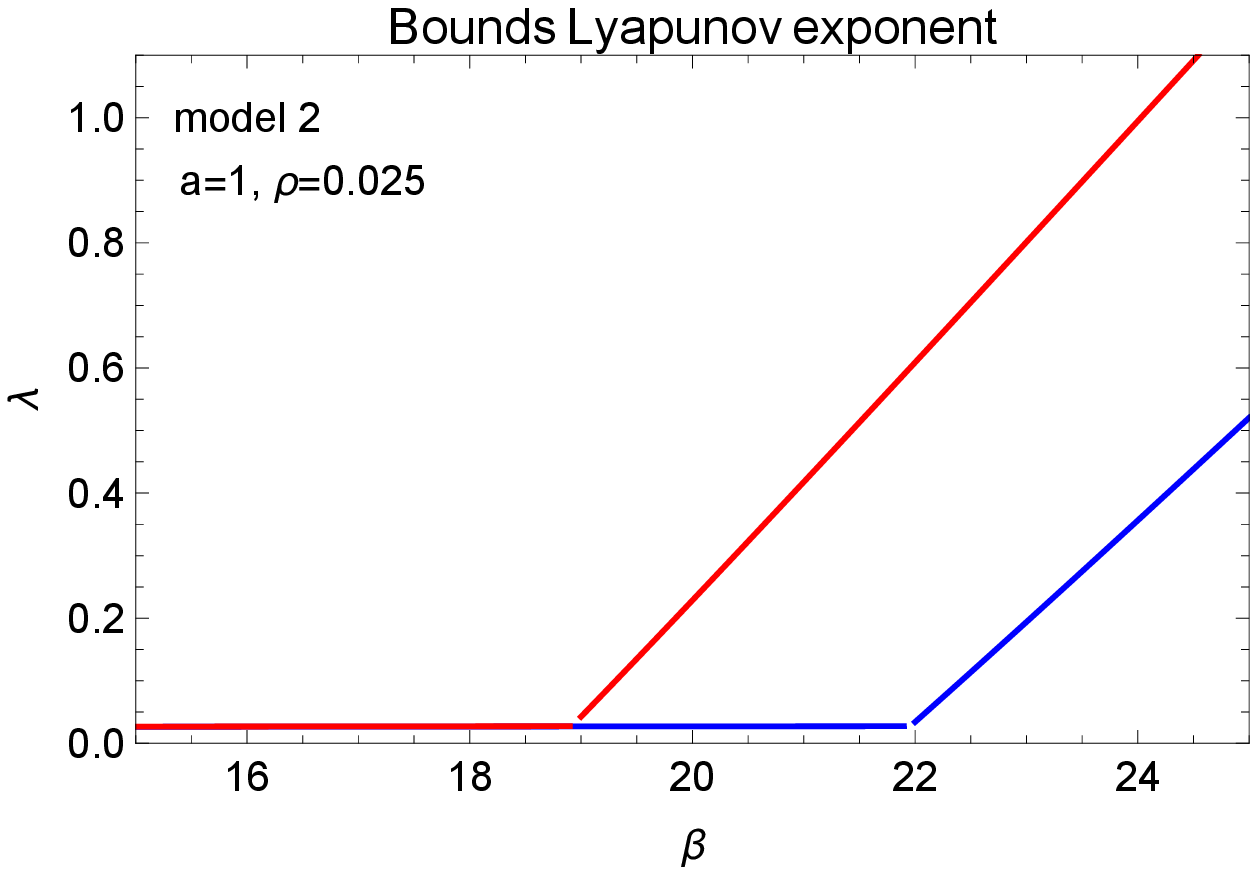}
    \caption{
%Left: the roots $x_{1,2}$ of the equation (\ref{xsteq}) vs $\beta$ for the lower (blue) %and upper (red) bounds, in their respective domains of applicability.
Lower (blue) and upper (red) bounds on the Lyapunov exponent $\lambda(\rho,\beta)$ vs $\beta$ in the models with stationary OU process (model 1) and OU process started at zero (model 2). The bounds have discontinuous derivative at the transition points $\beta_{\rm cr}^{\rm low}=-\frac{\log\rho}{k(a)}=10.0274$ for the lower bound and by  $\beta_{\rm cr}^{\rm hi}=-\frac{\log\rho}{\lambda_0(a)}=9.986$ for the 
upper bound (for model 1). 
Model parameters $a=1.0, \rho=0.025$.
}
\label{Fig:Lam}
 \end{figure}

\begin{table}[t!]
\caption{\label{Table:vdW} 
The solutions $x_{1,2}$ of the equation (\ref{xsteq}) 
for the vdW approximation corresponding to the lower bound 
around the transition point $\beta_{\rm cr}^{\rm low}(\rho)=10.0274$. 
%They have the physical interpretation of gas and liquid particle number densities, respectively. 
The Lyapunov exponent $\lambda_{\rm vdW}$ corresponds to the blue curve in the left panel of Fig.~\ref{Fig:Lam}.
The gas-liquid transition temperature is shown in red. 
At the transition temperature the gas and liquid phases are in equilibrium. 
Model parameters $a=1,\rho=0.025$. }
\begin{center}
\begin{tabular}{|c|ccc|}
\hline
$\beta$ & $x_1$  & $x_2$ & $\lambda_{vdW}$ \\
\hline
\hline
9.0 & 0.029 & 0.914 & 0.0271 \\
9.9 & 0.030 & 0.966 &  0.0273 \\
10.0 & 0.030 & 0.969 &  0.0274 \\
{\color{red} 10.0274} & {\color{red} 0.030} & {\color{red} 0.970} & {\color{red} 0.0274}  \\
10.03  &  0.030  & 0.970  &   0.0283 \\
10.1 & 0.030 & 0.972 &  0.0525 \\
10.2 & 0.030 & 0.974 &  0.0874  \\
11.0 & 0.031 & 0.986 &  0.3706  \\
\hline
\end{tabular}
\end{center}
\end{table}

A similar transition takes place for all $\rho<e^{-2}$. The transition takes place as $\beta$ crosses the transition curve $\beta_{\rm cr}^{\rm low,hi}$. The plot of these transition curves are shown in Fig.~\ref{Fig:curve} in variables $(\rho,1/\beta)$ for $a=0.5,1,2$. The lower/upper bound transition curves are shown as blue/red curves. 

For the stationary O-U process (model 1) the curves are too close to be distinguishable, 
due to the closeness of the parameters $\lambda_0(a),k(a)$ noted in Table~\ref{Table:0b}.
The relative difference of these parameters is less than 0.04 over the entire range of values of $a>0$. It reaches its maximum at $a=9.4$. For $a<1$ the relative difference is always below 0.004. Thus the bounds are sufficiently constraining to determine the growth rate 
for model 1, up to a small error, for most practical applications.

\begin{figure}[b]
    \centering
   \includegraphics[width=3.2in]{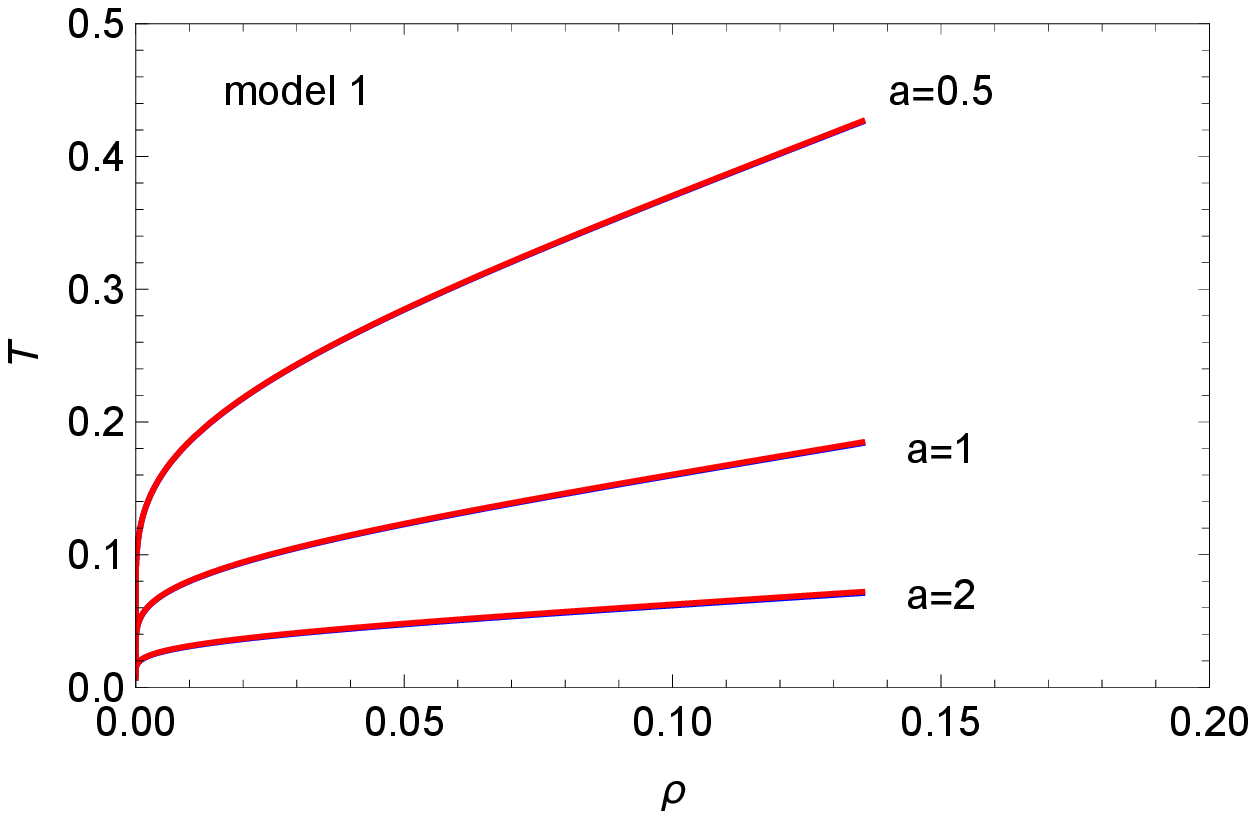}
   \includegraphics[width=3.2in]{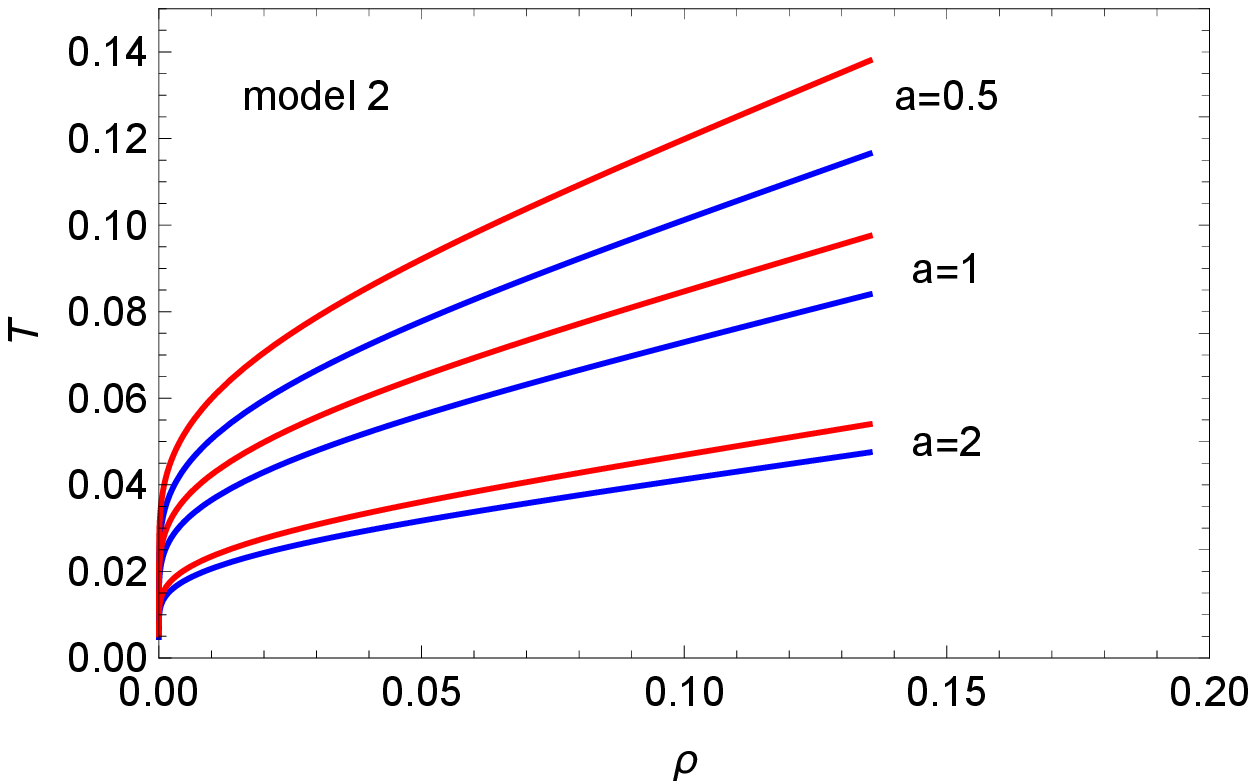}
    \caption{
Phase transition curves in variables $(\rho,T)$ with $T=1/\beta$
for the van der Waals models giving the lower/upper bounds on the Lyapunov exponent in the stochastic growth model with stationary OU process (left),
and in the model with OU process started at zero (right). 
They are shown as blue/red curves. The three
pairs of curves correspond to $a=0.5,1,2$ (from top to bottom).}
\label{Fig:curve}
 \end{figure}

\section{Numerical tests}
\label{sec:5}

The stochastic growth process considered here
\begin{equation}\label{Bprocess}
B_{i+1} = (1 + \rho e^{Z_i -\frac12 var(Z_i)} ) B_i\,,\quad B_1 =1
\end{equation}
can be simulated exactly using the recursive construction of the OU process 
$Z_{i+1} = e^{-\gamma \tau} Z_i + \sqrt{\Delta G} \varepsilon_i$ with 
$\Delta G = \frac{\sigma^2}{2\gamma}(1-e^{-2\gamma\tau})$ and 
$\varepsilon_i = N(0,1)$ iid random numbers. 

The initial condition is $Z_1=N(0,\frac{\sigma^2}{2\gamma})$ for the stationary OU process and $Z_1=0$ for the OU process started at zero.

In the deterministic limit $\sigma=0$ the terminal value of the process has an exponential growth $B_n=(1+ \rho)^{n-1}$. For non-zero volatility the expectation
$\mathbb{E}[B_n]$ increases with $\sigma$. Table \ref{Table:MC} shows 
$\mathbb{E}[B_n]$ and $Var(B_n)$ vs $\sigma$ for a simulation with 
$\rho=0.025,\gamma=\{0.1, 0.25, 0.5\}$ with $n=100$ time steps of length $\tau=0.01$.
Both moments have a smooth increase up to a threshold value of $\sigma$, after which it explodes to very large values. The threshold value increases with $\gamma$.

%\begin{figure}[t]
%    \centering
%   \includegraphics[width=1.8in]{numericalSimulations/expBmodel1p1.png}
%   \includegraphics[width=1.8in]{numericalSimulations/expBmodel1p2.png}
%    \caption{
%Monte Carlo simulation of $\mathbb{E}[B_n], var(B_n)$ vs $\sigma$ for the stochastic growth process (\ref{Bprocess}) driven by a stationary OU process $Z_i$ with mean reversion $\gamma=0.1$ and multiplier $\rho=0.025$. The time discretization has $n=100$ time steps of $\tau=0.01$. The MC simulation used $N_{MC}=10^3$ paths.
%The deterministic limit is $B_{100}=(1+\rho)^{n-1}=11.5256$.}
%\label{Fig:MC1}
% \end{figure}

\begin{table}[t!]
\caption{\label{Table:MC} 
Monte Carlo simulation of $\mathbb{E}[B_n], var(B_n)$ vs $\sigma$ for the stochastic growth process (\ref{Bprocess}) driven by a stationary OU process $Z_i$ with 
mean reversion $\gamma=\{0.1, 0.25, 0.5\}$ and multiplier $\rho=0.025$. 
The time discretization has $n=100$ time steps of $\tau=0.01$. 
The MC simulation used $N_{MC}=10^3$ paths.
The deterministic limit value is $B_{100}=(1+\rho)^{n-1}=11.5256$.}
\begin{center}
\begin{tabular}{c|cc|cc|cc}
\hline
 & \multicolumn{2}{c|}{$\gamma=0.1$}
 & \multicolumn{2}{c|}{$\gamma=0.25$}
 & \multicolumn{2}{c}{$\gamma=0.5$} \\
\cline{2-7}
$\sigma$ & $\mathbb{E}[B_n]$ & $Var(B_n)$ 
               & $\mathbb{E}[B_n]$ & $Var(B_n)$
               & $\mathbb{E}[B_n]$ & $Var(B_n)$ \\
\hline
\hline
0.01 & 11.525 & 0.377 & 11.519 & 0.149 & 11.528 & 0.069 \\
0.05 & 11.972 & 11.991 & 11.753 &  3.868 & 11.587 & 1.602 \\
0.10 & 13.513 & 90.30 & 12.019 & 19.08 & 11.978 & 8.426 \\
0.15 & 15.307 & 482.18 & 13.116 & 65.61 & 12.390 & 20.689 \\
0.20 & 27.73 & 7734.2 & 15.667 & 578.7 & 12.803 & 48.22 \\
0.25 & 575.9 & $5.2\cdot 10^7$ & 17.636 & 869.8 & 13.68 & 151.79 \\
0.30 & 8128.5 & $2.7\cdot 10^{10}$ & 24.61 & 7715.5 & 18.319 & 5,704.5 \\
0.35 & $7.2\cdot 10^5$ & $2.9\cdot 10^{12}$ 
        & 76.09 & $5.7\cdot 10^5$ & 19.54 & 2,008.6 \\
0.40 & $1.3\cdot 10^{10}$ & $1.8\cdot 10^{23}$ 
        & $5.6\cdot 10^4$ & $3.1\cdot 10^{12}$ & 22.27 & 4,558.6 \\
0.45 & $1.2\cdot 10^{11}$ & $1.3\cdot 10^{25}$ 
        & $3.1\cdot 10^3$ & $9.0\cdot 10^9$ & 26.49 & 5,606.1 \\
0.50 & $2.2\cdot 10^{11}$ & $4.2\cdot 10^{25}$ 
        & $8.6\cdot 10^3$ & $6.2\cdot 10^{10}$ & 41.67 & $6.8\cdot 10^4$ \\
\hline
\end{tabular}
\end{center}
\end{table}

The simulation results are shown also in the left plots of Figure~\ref{Fig:MC}, for a wider range of values of $\gamma$. The plots show $\mathbb{E}[B_n]$ with one-standard deviation error bars $\pm \frac{1}{\sqrt{N_{MC}}} stdev(B_n)$. As $\sigma$ increases, we note first an explosion of the standard deviation, followed by a similar explosion in the central value of the expectation itself. The explosion values of $\sigma$ increase with the mean-reversion parameter $\gamma$. 

The right plots in Figure~\ref{Fig:MC} show the growth rates $\lambda_n=\frac{1}{n}
\log\mathbb{E}[B_n]$ vs $\sigma$, comparing with the theoretical lower bound of Proposition~\ref{prop:bounds}
$\lambda(\rho,\beta) \geq \lambda_{vdW}^{\rm low}$ (blue curve), expressed in terms of $\sigma$ using $\beta = \frac12 \sigma^2 n^2 \tau$. The bound has a discontinuous derivative at the transition point $\beta_{cr}^{\rm low}= 
- \frac{\log\rho}{k(a)}$, followed by a rapid increase in the Lyapunov exponent.
This explains the rapid increase of the expectation $\mathbb{E}[B_n]$ for $\sigma$ above a certain transition value, given by $\sigma_{tr}=\sqrt{\frac{2\beta_{cr}^{\rm low}}{n^2\tau}}$. The transition value $\sigma_{tr}$ is shown as the blue vertical line in the
left plots of Figure~\ref{Fig:MC}.

\begin{figure}[t]
    \centering
   \includegraphics[width=2.5in]{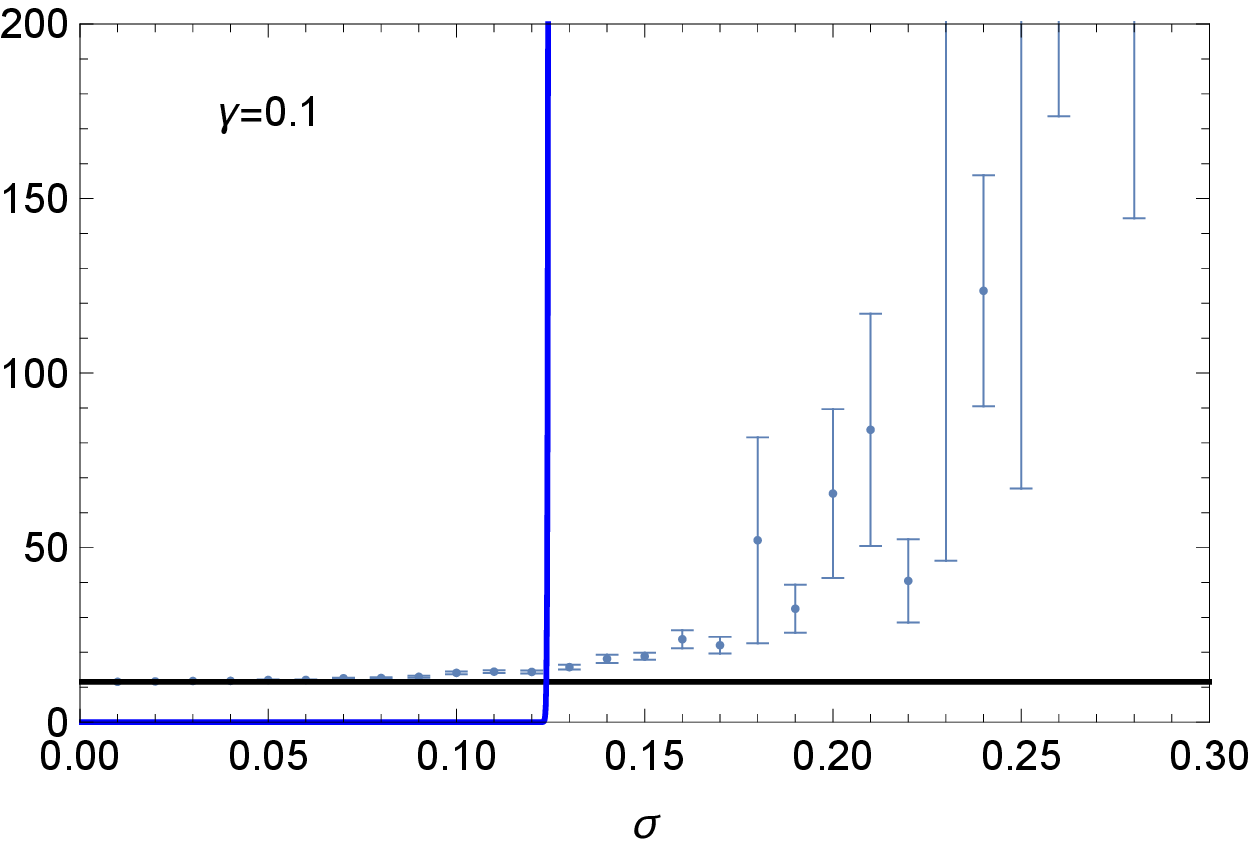}
   \includegraphics[width=2.5in]{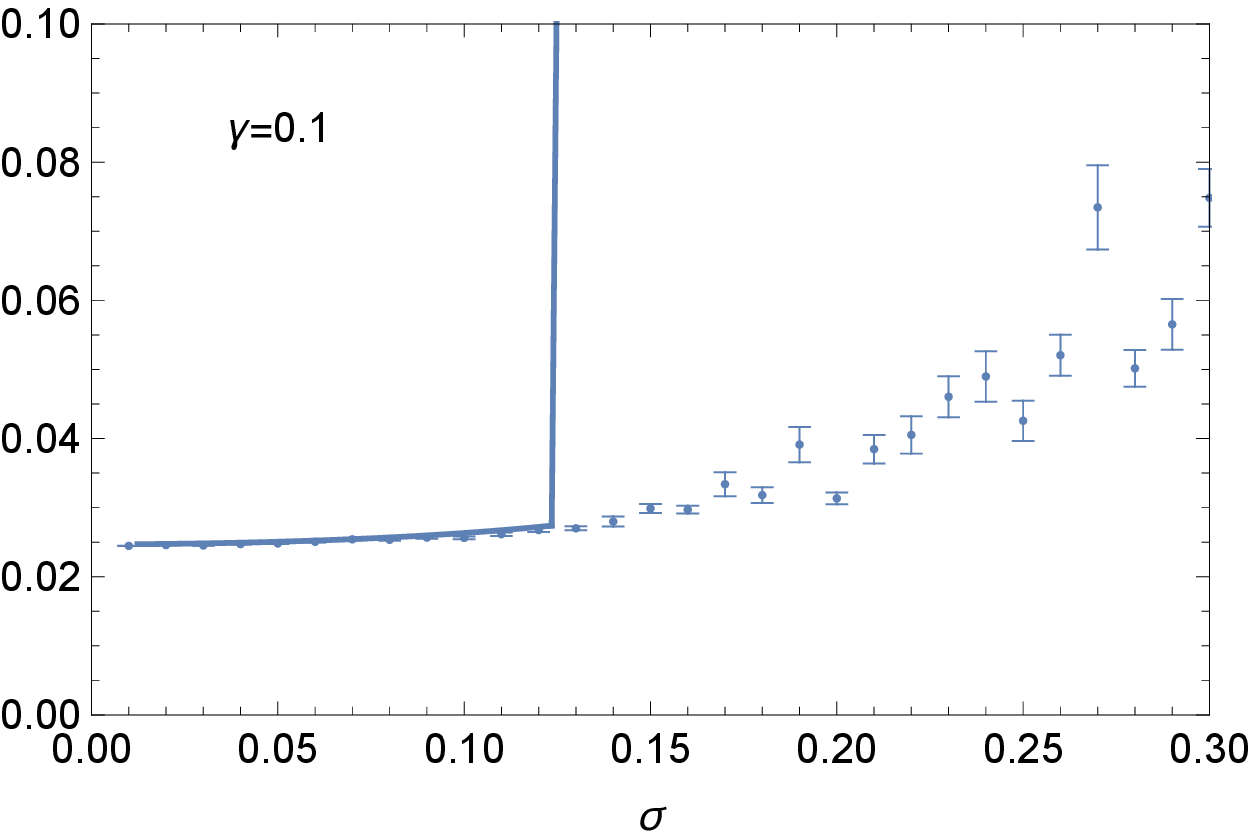}
   \includegraphics[width=2.5in]{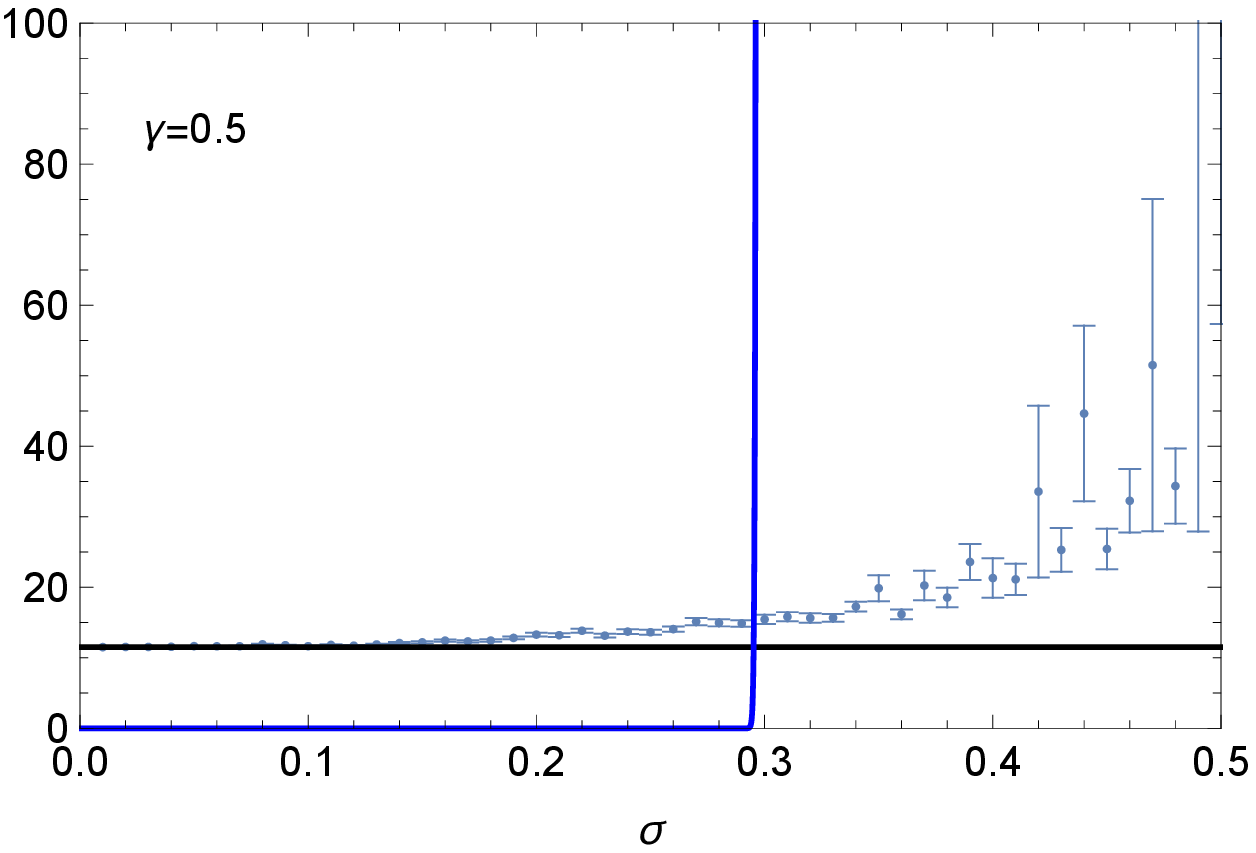}
   \includegraphics[width=2.5in]{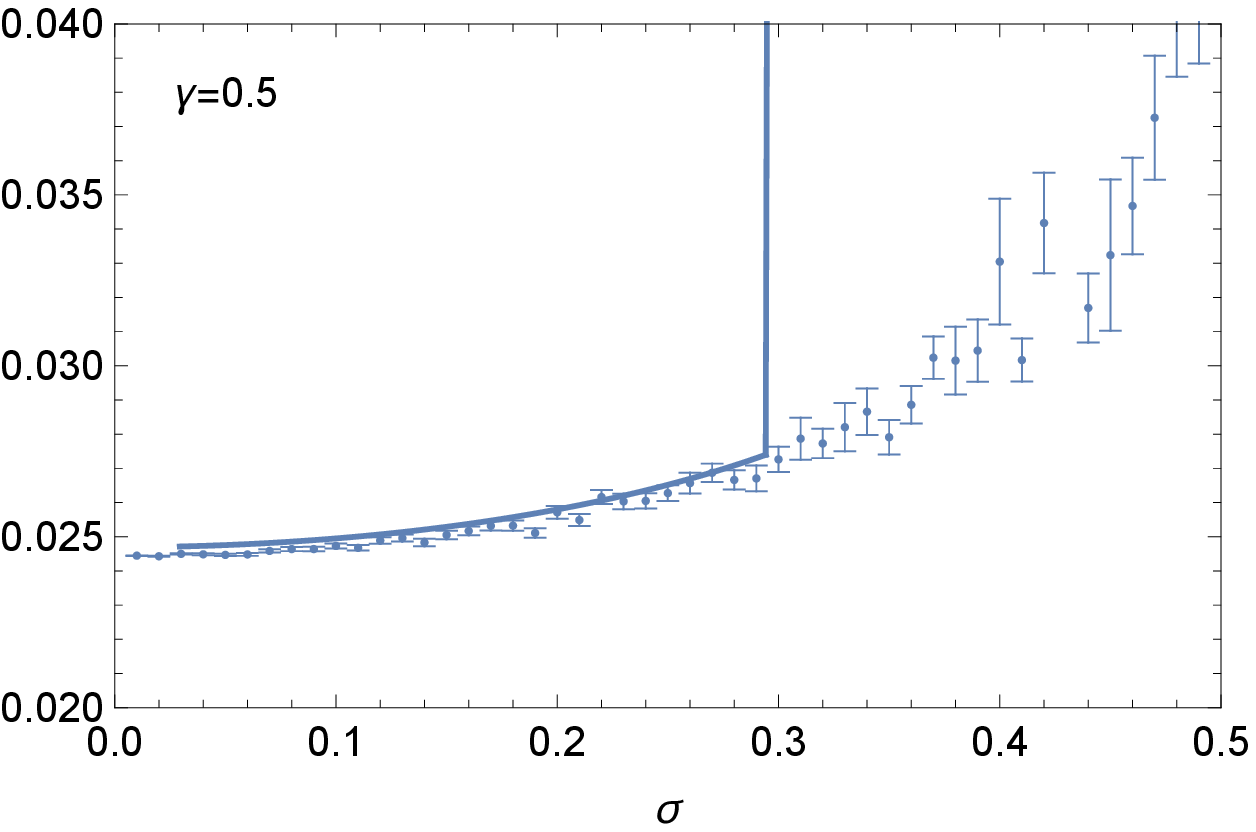}
   \includegraphics[width=2.5in]{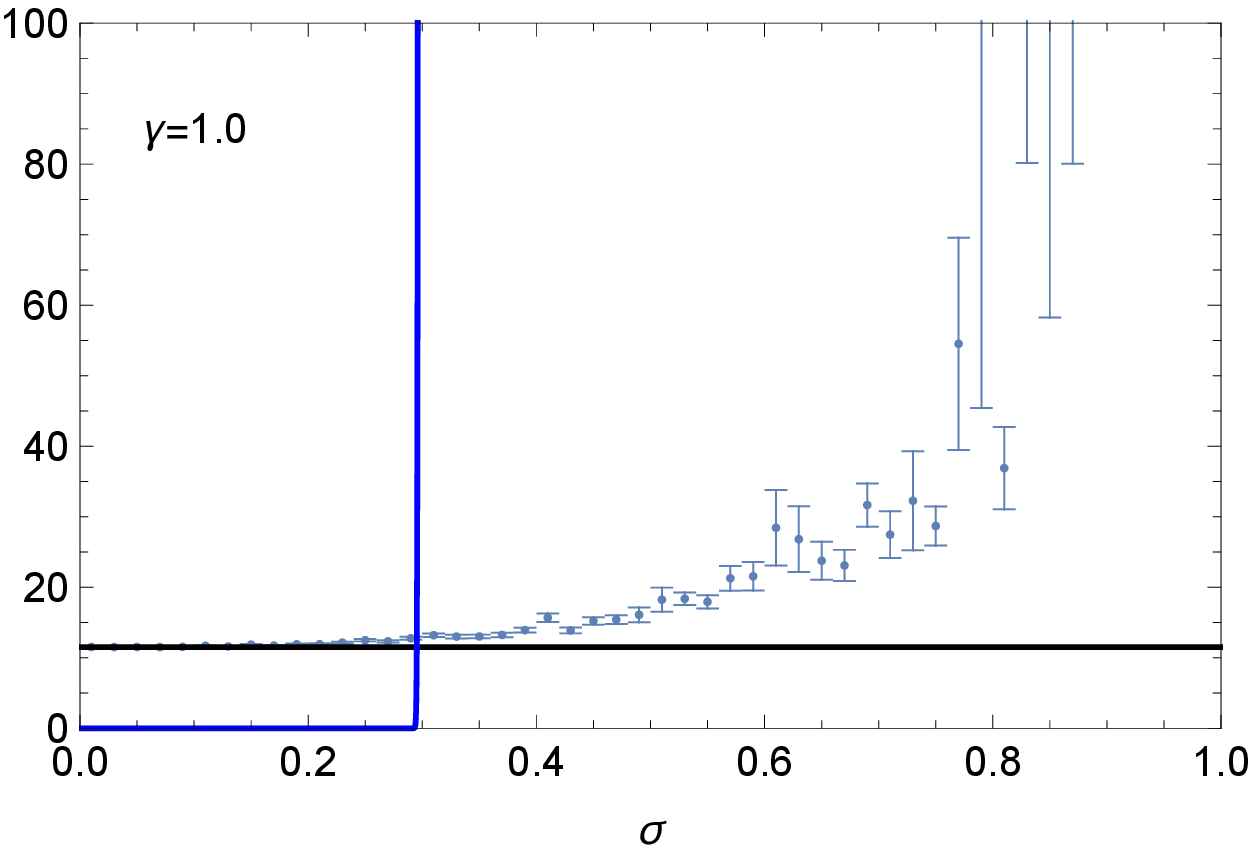}
   \includegraphics[width=2.5in]{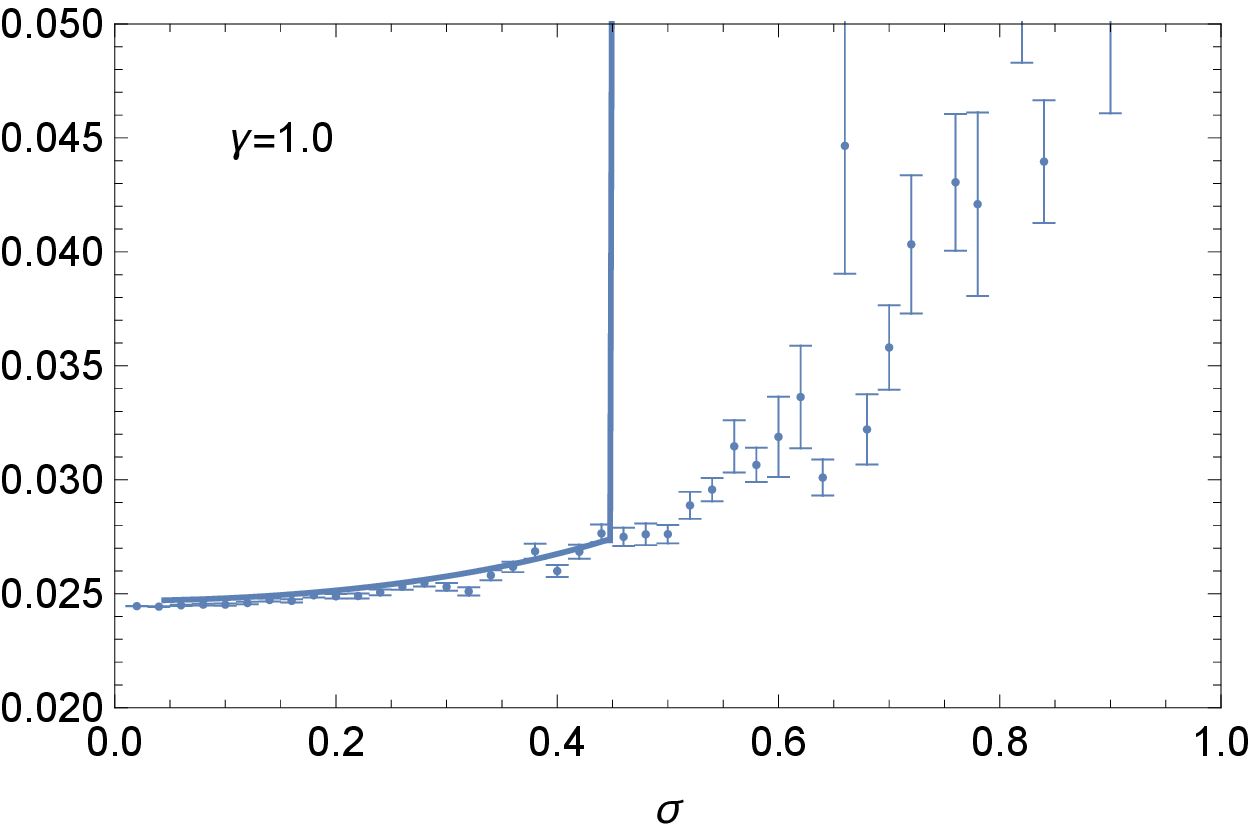}
    \caption{
Monte Carlo simulation of $\mathbb{E}[B_n]$ (left) and $\lambda_n = \frac{1}{n} \log \mathbb{E}[B_n]$ vs $\sigma$ (right) for the stochastic growth process (\ref{Bprocess}) driven by a stationary OU process. The mean reversion is as shown.
The remaining simulation parameters are the same as in Table~\ref{Table:MC}.}
\label{Fig:MC}
 \end{figure}

The remaining discrepancy between the MC simulation and the asymptotic (lower bound) Lyapunov exponent $\lambda(\rho,\beta)$ can be explained by a phenomenon noted in \cite{EDpaper}. The error estimate for $\mathbb{E}[B_n]$ is obtained from the MC estimate of $Var(B_n)$. However the estimate of $Var(B_n)$ from the MC sample is subject to large MC errors due to the explosion of the fourth moment $\mathbb{E}[B_n^4]$ at a smaller value of $\sigma$. This results in an underestimate of the MC error estimate for $\mathbb{E}[B_n]$ and an apparent discrepancy between the observed error
$|\mathbb{E}[B_n]_{th} - \mathbb{E}[B_n]_{MC}|$ and the MC error estimate 
$\frac{1}{\sqrt{N_{MC}}} \sqrt{Var(B_n)}$. 
This phenomenon was studied in \cite{EDpaper} on the example of the
stochastic growth process driven by a geometric Brownian motion, where all integer moments $\mathbb{E}[B_n^q]$ with $q=1,2,\cdots$ can be evaluated exactly by a recursion relation. This allows a direct comparison of the simulation error with the MC error estimate.
A similar study is not possible for the model considered here, where an exact evaluation of moments is not feasible for $n \sim 100$.

For sufficiently small $n$, an exact evaluation of the moment $\mathbb{E}[B_n]$ for the present model is possible, by direct evaluation of the sum in (\ref{sumrho}).
This corresponds to a computation of the grand partition function for a lattice gas with $n$
sites, and involves the summation of $2^n$ terms, corresponding to all possible
occupational configurations of the lattice gas. The complexity of this 
calculation grows very fast with $n$, and is feasible only for small $n$. 

Using this method we evaluated the expectation $\mathbb{E}[B_n]$ for $n=10$ with $\rho=0.025, \tau=1.0$ and several choices of $\gamma$ for a stationary OU process. 
For given $\gamma$, as $\sigma$ increases through a threshold value, the expectation $\mathbb{E}[B_n]$ increases rapidly. The last column in Table \ref{Table:1} 
shows the threshold value $\sigma_{\rm exp}$ for each $\gamma$.

This numerical experiment is similar to the simulation of $\mathbb{E}[ B_n]$ for the stochastic growth process driven by an Ornstein-Uhlenbeck process started at zero 
presented in Section 6 of \cite{JSP}. Figure 7 in \cite{JSP} shows the results of an exact numerical evaluation of the expectation $\mathbb{E}[B_n]$ vs $\sigma$ at fixed mean-reversion $\gamma=0.01, 0.05, 0.1$. From Figure 7 in \cite{JSP} one can see that the expectation $\mathbb{E}[x_n]$ has a rapid increase as the volatility parameter $\sigma$ crosses a threshold value $\sigma_{\rm exp}$. 

We compare the threshold value $\sigma_{\rm exp}$ from the numerical simulation 
with the 
theoretical bounds following from the mean-field approximation for the bounds on the Lyapunov exponent derived in Proposition~\ref{prop:bounds}.

For each value of the mean-reversion parameter $\gamma$ we compute 
the liquid-gas transition temperatures $T_{\rm cr}^{\rm vdW}(\rho)$ corresponding 
to the lower/upper mean-field bounds. They are computed from   
\begin{equation}\label{betacr}
\beta_{cr}^{\rm low} = - \frac{\log\rho}{k(a)}\,,\quad
\beta_{cr}^{\rm hi} = - \frac{\log\rho}{\lambda_0(a)}\,.
\end{equation}

The third and fourth columns of Table~\ref{Table:1} show the prediction for the bounds on the explosion values of the volatility $\sigma$, related to the transition temperatures (\ref{betacr}) as $\sigma_{\rm exp}^{\rm vdW} = \frac{1}{n} \sqrt{\frac{2\beta}{\tau}}$.

The general trend of the explosion volatility $\sigma_{\rm exp}$ obtained from the numerical simulation is reproduced by the theoretical van der Waals predictions from the upper and lower bounds on the Lyapunov exponent. The agreement of the numerical simulation with the theoretical bounds improves with increasing $a$, as expected from the convergence of the bounds in this limit noted in 
Section~\ref{sec:4.1}.
The remaining discrepancies may be explained by the small value of $n$ used in the
numerical simulation $n=10$, as the theoretical predictions are asymptotic
and hold exactly only in the $n\to \infty$ limit. 

\begin{table}[t!]
\caption{\label{Table:1} 
Numerical results for the explosion value of the volatility parameter
$\sigma_{\rm exp}$ at which the average $\mathbb{E}[ x_n]$ has an explosion
for given mean reversion $\gamma$. This is compared with the prediction
from the liquid-gas phase transition in the van der Waals theory for the lower
and upper bounds on the Lyapunov exponent.
The simulation used a lattice with $n=10$ sites and assumed a stationary OU process
with parameters $\rho=0.025, \tau=1$. Recall $a=\gamma\tau n$.}
\begin{center}
\begin{tabular}{c|c|cc|cc|c}
\hline
$\gamma$ & $a$ 
                  & $T_{\rm cr}^{\rm vdW-low}$ & $T_{\rm cr}^{\rm vdW-hi}$
                  & $\sigma_{\rm cr}^{\rm vdW-low}$ & $\sigma_{\rm cr}^{\rm vdW-hi}$ 
                  & $\sigma_{\rm exp}$ \\
\hline
\hline
0.01 & 0.1 & 1.3114 & 1.3114 & 0.1235 & 0.1235 & 0.14 \\
0.05 & 0.5 & 0.2310 & 0.2313 & 0.2942 & 0.2940 & 0.34 \\
0.1  & 1.0 &  0.0997 & 0.1001 & 0.4478 & 0.4469 & 0.52 \\
0.2  & 2.0 &  0.0385 & 0.0389 & 0.7210 & 0.7166 & 0.83 \\
0.3  & 3.0 &  0.0206 & 0.0210 & 0.9858 & 0.9758 & 1.15 \\
0.4  & 4.0 &  0.0128 & 0.0131 & 1.2508 & 1.2340 & 1.47 \\
0.5  & 5.0 &  0.0087 & 0.0090 & 1.5171 & 1.4931 & 1.83 \\
1.0  & 10.0 & 0.0024 & 0.0025 & 2.8631 & 2.8084 & 3.90 \\
\hline
\end{tabular}
\end{center}
\end{table}

\section{Discussion and comparison with the literature}
\label{sec:6}

As discussed in Sec.~\ref{sec:3}, the thermodynamical analog of the compounding process
$x_n$ is a one-dimensional lattice gas with $n$ sites. At each site there is at most one 
particle, and the particles interact by two-body attractive potentials. For the stochastic growth process driven by a stationary Ornstein-Uhlenbeck process the 
two-body potentials are
\begin{eqnarray}\label{eps2}
\varepsilon_{ij} = - \frac{c}{a n} 
\exp\left( - \frac{a}{n} |i-j| \right) \,.
\end{eqnarray}
This is obtained by substituting the scaling of the OU process parameters (\ref{scaling})
$\beta = \frac12 \sigma^2 \tau n^2, a=\gamma\tau n$ into the two-body interaction (\ref{beps1}). For purposes of comparison with the statistical mechanics literature 
to be discussed below we added the constant $c$. The stochastic growth process corresponds to $c=1$.
The lattice gas is at temperature $\beta$ and fugacity $\rho$. 

A similar system was considered by Kac and Helfand in Section 4 of \cite{KH}.
We present next a detailed comparison with their results. 
The authors of \cite{KH} consider a one-dimensional lattice gas with $n$ sites and 
exponential attractive interaction\footnote{We add subscripts
on the constants appearing here to avoid notational confusion with $\gamma$ used 
earlier for the mean-reversion of the OU process $Z_t$.}
\begin{eqnarray}\label{epsKH}
\varepsilon_{ij}^{\rm KH} = -\alpha_{KH} \gamma_{KH} e^{-\gamma_{KH} |i-j|} \,.
\end{eqnarray}
This is a lattice gas version of the Kac, Uhlenbeck, Hemmer system \cite{KUH}, 
which is a continuous one-dimensional gas of particles interacting by two-body exponential attractive  potentials.

Kac and Helfand take the thermodynamical 
limit $n\to \infty$, followed by the limit $\gamma_{KH} \to 0$.
The limits $n\to \infty, \gamma_{KH}\to 0$ can be also taken simultaneously,
provided that one keeps $n\gg \gamma_{KH}^{-1}$. This is called the van der Waals
limit. In this limit the system is 
described by a van der Waals equation of state, and has a liquid-vapor phase 
transition.  Keeping $\gamma$ finite, the system does not have a phase transition. 
The approach to the $\gamma\to 0$ limit is studied in further detail by Helfand \cite{Helfand}. These results are recovered in the Lebowitz, Penrose theory \cite{LP} 
which extends these properties to a more general class of potentials (Kac potentials) 
which include the exponential potential as a particular case.

The exponential factor in (\ref{eps2}) reproduces the Kac and Helfand interaction by identifying the range parameter as $\gamma_{KH}=\frac{a}{n}$. As $n\to \infty$ the range parameter goes to zero as required. The van der Waals condition $n\gg \gamma_{KH}^{-1}$ is satisfied provided that one takes $a \gg 1$.
Setting the overall factor $c = \alpha_{KH} a^2$ reproduces precisely the Kac, Helfand interaction (\ref{epsKH}). 

In conclusion, the general two-body interaction (\ref{eps2}) reduces to the Kac-Helfand interaction (\ref{epsKH}) by taking
\begin{eqnarray}\label{KHconversion}
c= \alpha_{KH} a^2\,,\qquad \gamma_{KH} = \frac{a}{n}\,,
\end{eqnarray}
Furthermore, the van der Waals limit is realized in the large mean-reversion limit $a \gg 1$.

The thermodynamics of the system with interaction (\ref{eps2}) can be obtained immediately
from the results of Section~\ref{sec:4} and applying the analogy to the lattice gas described
in Sec.~\ref{sec:3}. 
The result of Theorem \ref{thm:OU} is adapted by replacing $K(y,z) \to c K(y,z)$, and the results of 
Proposition~\ref{prop:bounds} are modified
by replacing $k(a) \to c k(a), \lambda_0(a) \to c \lambda_0(a)$. 
Furthermore, in Section~\ref{sec:4.2} it was noted that 
$\lambda_0(a), k(a)\simeq \frac{1}{a^2} + O(1/a)$ have the same leading asymptotics 
as $a\to \infty$.  Using $c(a)=\alpha_{\rm KH} a^2$ we get
\begin{equation}
\lim_{a\to \infty} c \lambda_0(a) = \lim_{a\to \infty} c k(a) =\alpha_{\rm KH}\,.
\end{equation}
Thus the upper and lower bounds converge to a common result as $a\to \infty$.

The thermodynamical pressure of the lattice gas with the Kac-Helfand interaction (\ref{epsKH}) becomes
\begin{equation}
p = - T \log(1-d) - \alpha_{KH} d^2
\end{equation}
which coincides precisely with the result obtained in \cite{KH} for the 
pressure of a van der Waals lattice gas (equation (4.11) in \cite{KH}).
The density $d$ is related to the fugacity $\rho$ by replacing $k \to \alpha_{\rm KH}$ in Eq.~(\ref{xsteq}). This gives $\rho = e^{-2\alpha_{KH}\beta d} \frac{d}{1-d}$,
which reproduces equation (4.10) in \cite{KH}\footnote{Kac and Helfand \cite{KH}
express this result
in terms of a reduced temperature parameter $\nu = \frac14 \beta \alpha_{\rm KH}$.}.
The system has a liquid-gas phase transition at points along the curve
\begin{equation}
\log \rho +\alpha_{KH} \beta=0
\end{equation}
which ends at the critical point $\beta_c = \frac{2}{\alpha_{KH}}$.

Going beyond the $a\to \infty$ limit corresponds to a lattice gas with exponential interactions of range $\gamma^{-1} = n/a$ which is still very large, in the $n\to \infty$ limit, but of comparable size with the lattice size $n$. 
An expansion in the range of an exponential interaction was considered by Kac and Thompson in \cite{Kac1969}. In our language this corresponds to including subleading terms of $O(1/a)$. The result of Proposition~\ref{prop:bounds} gives very constraining
upper and lower bounds on the thermodynamical pressure which suggests that the subleading corrections of $O(1/a)$ are small. 
A detailed study of these corrections in the context of the stochastic growth process considered in this paper will be performed elsewhere.

\section{Summary and conclusions}

We studied in this paper the growth rate of the average of a stochastic growth
process with growth rate proportional to the exponential of an Ornstein-Uhlenbeck process.
The asymptotic growth rate can be computed exactly in the limit of a very large number of time steps $n\to \infty$, by appropriately rescaling of the model 
parameters - the mean growth rate $\rho$, the volatility of the OU process $\sigma$, and the mean-reversion $\gamma$ - with $n$. 

Using large deviations theory methods we derived an exact representation for the asymptotic growth rate (Lyapunov exponent) as the solution of a variational problem. The method of proof is similar to that used in \cite{PZ} for a similar stochastic growth process, where the growth rate is proportional to the exponential of a standard Brownian motion.

We derive upper and lower bounds on the asymptotic growth rate (Lyapunov exponent) following from the variational problem. 
For the stochastic growth process driven by a stationary Ornstein-Uhlenbeck process the bounds are very predictive and constrain the growth rate $\lambda$ with an error below 
4\% for all parameter values. 
The bounds are given by solutions to the Curie-Weiss model (mean-field approximation), and have a discontinuous behavior at a certain value of the model parameters. More precisely, the asymptotic growth rate has a rapid increase as the volatility exceeds a minimum value. This agrees with a numerical explosion of the expected value of the growth process observed in numerical simulations.

The stochastic growth process can be mapped to a one-dimensional lattice gas with 
exponential attractive interactions, and the discontinuous behavior of the growth rate corresponds to a liquid-gas phase transition in the lattice gas. The lattice gas has a 
liquid-vapor phase transition, which is seen in the random multiplicative model as an explosion of the expectation $\mathbb{E}[ x_n]$. This phenomenon is related to the 
rapid increase of the pressure of the equivalent lattice gas as the temperature drops below the first order phase transition temperature. 

A similar one-dimensional lattice gas with attractive exponential interactions 
$\varepsilon_{ij} = -\alpha \gamma e^{-\gamma |i-j|}$ was 
considered in a different context by Kac and Helfand \cite{KH},
as a discrete version of a one-dimensional system studied by Kac, Uhlenbeck and 
Hemmer \cite{KUH}. Our results reproduce the Kac-Helfand results 
in the van der Waals limit, when the interaction range $1 \ll \gamma^{-1} \ll n$ is much smaller than the lattice size $n$.
The bounds of Proposition \ref{prop:bounds} imply that the properties of the Kac-Helfand system are well approximated by the van der Waals theory in the thermodynamical limit, 
even for  interaction range comparable with the lattice size $\gamma^{-1} \simeq n$. 

\appendix

\section{Proof of Theorem~\ref{thm:OU}}
\label{sec:app}

\begin{proof}[Proof of Theorem~\ref{thm:OU}]
The expectation of $x_n$ is
\begin{eqnarray}\label{Mn}
M_n &=& \mathbb{E}[x_n] = \mathbb{E}\left[\prod_{i=0}^{n-1} \left(1 + \rho e^{Z_i - \frac12 \mbox{var}(Z_i)}\right)\right] \\
&=& 2^n \mathbb{E}[ \exp( \log\rho \sum_{i=0}^{n-1} Y_i + \sum_{i=0}^{n-1} Y_i (Z_i - \frac12 \mbox{var}(Z_i)) )] \nonumber
\end{eqnarray}
where we introduced $\{Y_i\}_{i=0}^{n-1}$ iid binomial random variables
taking values $Y_1=\{0,1\}$ with equal probabilities. 

The values of the OU process $Z_i$ can be written as sums of iid random variables 
$V_j=N(0,\Delta G)$ as
\begin{eqnarray}\label{Zi}
Z_i = z^{i} Z_0 + \sum_{j=0}^{i-1} z^{i-j-1} V_j
\end{eqnarray}
with $z=e^{-\gamma \tau}$ and $\Delta G = G(t_{i+1} - t_i) = (1-z^2) \frac{\sigma^2}{2\gamma}$ with $G(t) =  \frac{\sigma^2}{2\gamma} (1 - e^{-2\gamma t})$, see (\ref{Gdef}). 
The relation (\ref{Zi}) follows by repeated application of (\ref{OUsol}), which gives $Z_{i+1} = e^{-\gamma \tau} Z_i + \sqrt{\Delta G} \varepsilon$ with
$\varepsilon=N(0,1)$.
For example, we have
\begin{eqnarray}
Z_1 &=& z Z_0 + V_0 \\
Z_2 &=& z^2 Z_0 + z V_0 + V_1 \,,
\end{eqnarray}
For generality we keep $Z_0$ as an unspecified random variable. 
Taking $Z_0=N(0,\frac{\sigma^2}{2\gamma})$ gives the stationary OU process, 
and $Z_0=0$ gives the model with OU process started at zero.

Noting that $\mbox{var}(V_i) = \Delta G = (1-z^2) \frac{\sigma^2}{2\gamma}$ 
we get from (\ref{Zi}) 
\begin{eqnarray}
\mbox{var}(Z_i) &=& z^{2i} \mbox{var}(Z_0) + (1 + z^2 + \cdots + z^{2(i-1)})(1-z^2)
\frac{\sigma^2}{2\gamma} \\
&=& z^{2i} \mbox{var}(Z_0) + (1 - z^{2i}) \frac{\sigma^2}{2\gamma} \,.\nonumber
\end{eqnarray}
We observe that if $Z_0 \sim N(0,\frac{\sigma^2}{2\gamma})$ then all variables
$Z_i$ have the same variance $\mbox{var}(Z_i) = \frac{\sigma^2}{2\gamma}$.
This corresponds to a stationary Ornstein-Uhlenbeck process. However, for
generality we will write the expressions below allowing for non-uniform 
$\mbox{var}(Z_i)$.

Substituting the expression (\ref{Zi}) into (\ref{Mn}), and exchanging the summation over $i,j$ we have
\begin{eqnarray}
\sum_{i=0}^{n-1} Y_i Z_i &=& Z_0 \sum_{i=0}^{n-1} z^{i} Y_i +  
         \sum_{i=0}^{n-1} Y_i \sum_{j=0}^{i-1} z^{i-j-1} V_j \\
&=& Z_0 \sum_{i=0}^{n-1} z^{i} Y_i +  
        \sum_{j=0}^{n-2} V_j \sum_{i=j+1}^{n-1} z^{i-j-1} Y_i + Y_0 V_0\,.
\end{eqnarray}

Let us consider in more detail the argument of the exponential in Eq.~(\ref{Mn}) which we denote as
\begin{eqnarray}
S_n &=& \log \rho \sum_{i=0}^{n-1} Y_i + \sum_{i=0}^{n-1} Y_i Z_i 
     - \frac12 \left(\sum_{i=0}^{n-1} \mbox{var}(Z_i) Y_i\right) \nonumber \\
&=& \log\rho \sum_{i=0}^{n-1} Y_i + Z_0 \sum_{i=0}^{n-1} z^{i} Y_i 
 +  \sum_{j=0}^{n-2} V_j \sum_{i=j+1}^{n-1} z^{i-j-1} Y_i 
 - \frac12 \left(\sum_{i=0}^{n-1} \mbox{var}(Z_i) Y_i \right) + Y_0 V_0 \nonumber
\end{eqnarray}

Taking the expectation over $Z_0, V_i$ we get
\begin{eqnarray}
M_n &=& 2^n \mathbb{E}_{Y_i}\left[e^{\bar S_n + \frac12 Y_0 \mbox{var}(V_0)}\right] \\
\bar S_n &=&
\log\rho \sum_{i=0}^{n-1} Y_i
+\frac12 \mbox{var}(Z_0) (\sum_{i=0}^{n-1} z^{i} Y_i)^2 
+  \frac12 \mbox{var}(V_i) \sum_{j=0}^{n-2} 
   (\sum_{i=j+1}^{n-1}z^{2(i-j-1)} Y_i)^2 \nonumber \\
& &- \frac12 \left(\sum_{i=0}^{n-1} \mbox{var}(Z_i) Y_i\right) \nonumber
\end{eqnarray}

Under the assumed scaling (\ref{scaling}), we have
\begin{eqnarray}
\mbox{var}(Z_0) &=& \frac{\sigma^2}{2\gamma} = \frac{\beta}{a n} \\
\mbox{var}(V_i) &=& (1-e^{-2\gamma\tau}) \frac{\sigma^2}{2\gamma} =
\frac{2a}{n} \cdot \frac{\beta}{a n} = \frac{2\beta}{n^2} \,.
\end{eqnarray}
The last term in $\bar S_n$ is subleading in $n$ and can be neglected.
The same holds for the term $\frac12 \mbox{var}(V_0) Y_0$ which is subleading in the 
large $n$ limit.

Keeping only the leading terms we have
\begin{eqnarray}
M_n &=& 2^n \mathbb{E}_{Y_i}\left[\exp\left(
\log\rho \sum_{i=0}^{n-1} Y_i +\frac{\beta}{2an} (\sum_{i=0}^{n-1} z^{i} Y_i)^2 +  \frac{\beta}{n^2} \sum_{j=0}^{n-2} (\sum_{i=j+1}^{n-1} z^{i-j-1} Y_i)^2\right)\right] \nonumber
\end{eqnarray}

The argument of the exponential can be written as
\begin{eqnarray}\label{Sbar}
&& \log\rho \sum_{i=0}^{n-1} Y_i +\frac{\beta}{2an} (\sum_{i=0}^{n-1} z^{i} Y_i)^2 +  \frac{\beta}{n^2} \sum_{j=0}^{n-2}  (\sum_{i=j+1}^{n-1}z^{i-j-1}  Y_i)^2  \\
&& = n \left[ 
\log\rho \left(\frac{1}{n} \sum_{i=0}^{n-1} Y_i\right) +\frac{\beta}{2a} \left(\frac{1}{n}\sum_{i=0}^{n-1} z^{i} Y_i\right)^2 +  \beta \left( \frac{1}{n} \sum_{j=0}^{n-2}  (\frac{1}{n}\sum_{i=j+1}^{n-1} z^{i-j-1}  Y_i)^2 \right) \right]\nonumber
\end{eqnarray}

The Mogulskii theorem in large deviations theory (see Theorem 5.1.2 in \cite{DemboZeitouni})
says that $\mathbb{P}(\frac{1}{n} \sum_{i=1}^{[n\cdot]} Y_i \in \cdot )$ satisfies a sample path
large deviations principle on the space $L_\infty[0,1]$ with rate function
\begin{equation}
\int_0^1 I(g'(x)) dx\,,\quad I(x) = x \log x + (1-x) \log(1-x) + \log 2
\end{equation}
where $g(0)=0$, $g$ is absolutely continuous with bounded derivative $0 \leq g' \leq 1$,
and the rate function is $+\infty$ otherwise. 
We have informally
\begin{equation}
\mathbb{P}\left( \frac{1}{n} \sum_{i=1}^{[n x]}  Y_i \sim  g(x), 0 \leq x \leq 1 \right)
\sim \exp( - n \int_0^1 I(g'(x)) dx + o(n) )
\end{equation}

%%%%%%% 

Denoting $\frac{1}{n} \sum_{i=1}^{[nx]} Y_i=g(x)$, the contents of the square bracket 
in (\ref{Sbar}) can be written in the $n\to \infty$ limit as
\begin{eqnarray}\label{zeta}
\Lambda[g] := \log \rho g(1) + \frac{\beta}{2a} 
\left(\int_0^1 dg(x) e^{-ax} \right)^2 + 
\beta \int_0^1 dx \left( \int_x^1 dg(y) e^{-a(y-x)} \right)^2 \,.
\end{eqnarray}
Recall $z = e^{-\gamma \tau}$ and we used $\gamma\tau = \frac{a}{n}$ from the scaling (\ref{scaling}).

We are now in a position to evaluate the limit
$\lim_{n\to \infty} \frac{1}{n} \log \mathbb{E}[x_n]$.
This can be done using Varadhan's lemma: if $P_n$ satisfies a large
deviations principle with rate function $\mathcal{I}(x)$ on $\mathbb{X}$, and 
$F:\mathbb{X}\to \mathbb{R}$ is a bounded and continuous function, then
\begin{equation}
\lim_{n\to \infty} \frac{1}{n} \log \mathbb{E}[e^{nF(x)}] = \sup_{x\in \mathbb{X}} \{
F(x)  - \mathcal{I}(x) \}\,.
\end{equation}
For our case it is easy to check that for any $g\in L_\infty[0,1]\cup \mathcal{G}$,
$g\to \log\rho g(1) + \beta \int_0^1 K(y,z) g'(y) g'(z) dy dz $ is a bounded and continuous map.
Hence, by Varadhan's lemma, we conclude that the limit
$\lim_{n\to \infty} \frac{1}{n} \log \mathbb{E}[x_n]$ exists and is given by the solution
of the variational problem (\ref{LambdaOU}).

%In conclusion, we find that the Lyapunov exponent of the moment $M_n$
%exists and is given by the variational problem of Theorem~\ref{thm:OU}.

The functional $\Lambda[g]$ simplifies further when expressed in terms of the
function $f(x)=g'(x)$.
Consider the sum of the second and third terms in Eq.~(\ref{zeta}). 
In the statistical mechanics interpretation they have the meaning
of minus interaction energy
\begin{equation}
 -U = \frac{1}{2a} \left(\int_0^1 dg(x) e^{-ax}\right)^2
  + \int_0^1 dx \left( \int_x^1 dg(y) e^{-a(y-x)} \right)^2 \,.
\end{equation}

The second integral can be written as
\begin{eqnarray}
&& \int_0^1 \left( \int_x^1 e^{-a(y-x)} dg(y) \right)^2 dx  \\
&& = \int_0^1 g'(y) g'(z) \theta(y-x)\theta(z-x) e^{-a(y-x)-a(z-x)} dx dy dz\nonumber \\
&& = \int_0^1 K_2(y,z) f(y) f(z) dx dy dz \nonumber 
\end{eqnarray}
with
\begin{eqnarray}\label{K2def0}
K_2(y,z) := \int_0^1 dx \theta(y-x)\theta(z-x) e^{-a(y-x)-a(z-x)} = 
\frac{1}{2a} (e^{-a|y-z|} - e^{-a(y+z)}) \,.
\end{eqnarray}

This yields
\begin{equation}
-U = \int_0^1 dy dz f(y) f(z) K(y,z)\nonumber
\end{equation}
with $K(y,z)=\frac{1}{2a} e^{-a|y-z|}$. This completes the proof of (\ref{LambdaOU}).

\end{proof}

%\section{An integral equation}

%\begin{figure}[t]
%    \centering
%   \includegraphics[width=5in]{KacIntegralEqp1.png}
%   \includegraphics[width=5in]{KacIntegralEqp2.png}
%    \caption{
%The Kac solution of the integral equation quoted in text.
%}
%\label{Fig:inteq}
% \end{figure}

\textbf{Acknowledgements.} I am grateful to Prof. Joel Lebowitz and the participants at the
Mathematical Physics seminar at Rutgers University, where a first version of these results was presented, for discussions and useful comments. I am indebted to Lingjiong Zhu for
collaboration on a related topic in \cite{PZ}.


\begin{thebibliography}{99}

\bibitem{Aristoff}
D.~Aristoff and L.~Zhu,
On the phase transition curve in a directed exponential random graph model,
\textit{Adv. Appl. Prob.} \textbf{50}(1), 272-301 (2018)

\bibitem{BlackK}
F.~Black and P.~Karasinski,
Bond and option pricing when short rates are lognormal,
\textit{Financial Analysts Journal} \textbf{47}(4), 52-59 (1991)

\bibitem{Caswell}
H.~Caswell,
\textit{Matrix Population Models: Construction, Analysis and Interpretation},
2nd Edition, Sinauer Associates Inc., Sunderland, MA, 2000.

\bibitem{JCohen}
J.E.~Cohen,
Long-run growth rates of discrete multiplicative processes in Markovian environments,
\textit{J.~Math.~Anal.~Appl.} \textbf{69}, 243 (1979)

\bibitem{CohenTL}
J.E.~Cohen,
Stochastic population dynamics in a Markovian environment implies Taylor's power law of fluctuation scaling,
\textit{Theoretical Population Biology} \textbf{93}, 30-37 (2014)

\bibitem{CourantHilbert}
R.~Courant and D.~Hilbert,
\textit{Methods of mathematical physics}, vol. 1,
Interscience Publishers, New York, 1953

\bibitem{Cumberland}
W.G.~Cumberland and Z.M.~Sykes,
Weak convergence of an autoregressive process used in modeling population growth,
\textit{J.~Appl.~Prob.} \textbf{19}, 450-455 (1982)

\bibitem{DemboZeitouni}
A.~Dembo and O.~Zeitouni,
\textit{Large deviations: Techniques and applications},
2nd Edition, Springer, New York 1998

\bibitem{Ellis}
R.~Ellis,
\textit{Entropy, Large Deviations and Statistical Mechanics}
(Classics in Mathematics), Springer, New York 2005.

\bibitem{GMS} 
G.~Gallavotti and S.~Miracle-Sole, 
Statistical mechanics of lattice systems, 
\textit{Comm.~Math.~Phys.} \textbf{5} 317 (1967).

\bibitem{Helfand}
E.~Helfand,
Approach to a phase transition in a one-dimensional system,
\textit{J.~Math.~Phys.} \textbf{5} 127 (1964).

\bibitem{Hemmer}
P.C.~Hemmer and J.L.~Lebowitz,
Systems with weak long-range potentials,
in \textit{Phase transitions and critical phenomena}, volume 5B,
Eds. C.~Domb and M.S.~Green, Academic Press (1976)

\bibitem{KacBarriers}
M.~Kac,
Random walk in the presence of absorbing barriers,
\textit{Ann. Math. Stat.}, \textbf{16}:62-67 (1945)

\bibitem{KH} 
M.~Kac and E.~Helfand, 
Study of several lattice systems with
long-range forces, 
\textit{J.~Math.~Phys.} \textbf{4} 1078 (1963).

\bibitem{KUH} 
M.~Kac, G.~E.~Uhlenbeck and P.~C.~Hemmer, 
On the van der Waals theory of the vapor-liquid equilibrium: I. 
Discussion of a one-dimensional model, 
\textit{J.~Math.~Phys.} \textbf{4} 216 (1963).

\bibitem{Kac1969}
M.~Kac and Colin J.~Thompson,
Critical behavior of several lattice models with long-range interaction,
\textit{J.~Math.~Phys.} \textbf{10} 1373 (1969).


\bibitem{LP} 
J.~Lebowitz and O.~Penrose, 
Rigorous treatment of the van der
Waals-Maxwell theory of the liquid-vapor transition, 
\textit{J.~Math.~Phys.} \textbf{7} 98 (1966).

\bibitem{Lebowitz1971}
J.~Lebowitz,
Some Exact Results in Equilibrium and Non-Equilibrium Statistical Mechanics. In \textit{Proceedings of the Advanced School for Statistical Mechanics and Thermodynamics}, number 7 in \textit{Lecture Notes in Physics},
University of Texas at Austin, 1971. Springer-Verlag. 
%Lectures in Statistical Methods.

\bibitem{Lebowitz1971a}
J.~Lebowitz and O.~Penrose, 
Rigorous Treatment of Metastable States in the van der Waals-Maxwell Theory. \textit{Journal of Statistical Physics}, \textbf{3}, 211 (1971)

\bibitem{JSP} 
D.~Pirjol, 
Emergence of heavy-tailed distributions in a random
multiplicative model driven by a Gaussian stochastic process, 
\textit{J.~Stat.~Phys.} \textbf{154} 781-806 (2014).

\bibitem{RMP2} 
D.~Pirjol, 
Long-run growth rate in a random multiplicative model, 
\textit{J.~Math.~Phys.} \textbf{55} 083305 (2014); arXiv:1503.02168[math-ph].

\bibitem{PZ} 
D.~Pirjol, L.~Zhu, 
On the growth rate of a linear stochastic recursion with Markovian dependence, 
\textit{J.~Stat.~Phys.} \textbf{160} 1354-1388 (2015).

\bibitem{EDpaper}
D.~Pirjol,
Eurodollar futures pricing in log-normal interest rate models in discrete time,
\textit{Applied Mathematical Finance} \textbf{23}(6) 445-464 (2017)

\bibitem{Radin}
C.~Radin and M.~Yin,
Phase transitions in exponential random graphs,
\textit{Annals of Applied Probability} \textbf{23}, 2458-2471 (2013)

\bibitem{Ruelle} 
D.~Ruelle, 
Statistical mechanics of a one-dimensional lattice gas, 
\textit{Comm.~Math.~Phys.} \textbf{9} 267-346 (1968).

\bibitem{Tuckwell} 
H.C.~Tuckwell, 
A study of some diffusion models of population growth, 
\textit{Theor.~Pop.~Biol.} \textbf{5} 345-357 (1974).

\bibitem{Tuljapurkar} 
S.D.~Tuljapurkar, 
\textit{Population dynamics in variable environments. }
In: \textit{Lecture Notes in Biomathematics}, vol. 85. 
Springer Verlag, Berlin, New York.

\end{thebibliography}
\end{document}